\documentclass[a4paper, oneside, svgnames]{amsart}
\usepackage{geometry}

\usepackage{amsmath}
\usepackage{amsfonts}
\usepackage{amssymb}
\usepackage{amsthm}
\usepackage{mathtools}
\usepackage{mathrsfs}
\usepackage{stmaryrd}
\usepackage{dsfont}
\usepackage[makeroom]{cancel}

\usepackage{tikz-cd}
\usepackage{graphicx}

\usepackage{caption}
\usepackage{subcaption}
\usepackage{comment}
\usepackage{enumerate}
\usepackage[textsize=tiny,textwidth=1.08in]{todonotes}

\usepackage{url}
\usepackage{xcolor}
\usepackage[colorlinks, linkcolor=DarkBlue, citecolor=cyan, backref=page]{hyperref} 

\numberwithin{equation}{section}

\newtheorem{proposition}{Proposition}[section]
\newtheorem{lemma}[proposition]{Lemma}
\newtheorem{theorem}[proposition]{Theorem}
\newtheorem{corollary}[proposition]{Corollary}

\theoremstyle{definition}
\newtheorem{remark}[proposition]{Remark}
\newtheorem{definition}[proposition]{Definition}

\DeclareMathOperator{\Bl}{Bl}
\DeclareMathOperator{\End}{End}
\DeclareMathOperator{\Id}{Id}
\DeclareMathOperator{\tr}{tr}

\DeclareMathOperator{\GL}{GL}
\DeclareMathOperator{\Aut}{Aut}

\DeclareMathOperator{\pr}{pr}

\DeclareMathOperator{\Scal}{Scal}
\DeclareMathOperator{\vol}{vol}

\newcommand{\C}{\mathbb{C}}

\renewcommand{\P}{\mathbb{P}}
\newcommand{\R}{\mathbb{R}}

\newcommand{\ddb}{i\partial \bar\partial}
\newcommand{\dbar}{\bar\partial}

\newcommand{\rk}{\mathrm{rk}}
\newcommand{\Vol}{\mathrm{Vol}}

\newcommand{\db}{\bar \partial}

\newcommand{\cB}{\mathcal{B}}
\newcommand{\cE}{\mathcal{E}}

\newcommand{\cH}{\mathcal{H}}
\newcommand{\cJ}{\mathcal{J}}

\newcommand{\cU}{\mathcal{U}}
\newcommand{\cV}{\mathcal{V}}

\newcommand{\cX}{\mathcal{X}}

\newcommand{\mfk}{\mathfrak{k}}

\newcommand{\mfh}{\mathfrak{h}}

\makeatletter
\define@key{todonotes}{AO}[]{%
    \setkeys{todonotes}{color=SkyBlue}}%
\makeatother 

\newcommand{\m}[1]{\mathcal{#1}}
\newcommand{\bb}[1]{\mathbb{#1}}

\newcommand{\mrm}[1]{\mathrm{#1}}
\newcommand{\mf}[1]{\mathfrak{#1}} 
\newcommand{\scr}[1]{\mathscr{#1}}

\newcommand{\dd}{\mathrm{d}}
\newcommand{\del}{\partial}
\newcommand{\delbar}{\bar\partial}
\newcommand{\Ok}[1]{O\left(k^{#1}\right)}
\newcommand{\Fut}{\mathrm{Fut}}

\DeclarePairedDelimiter\norm{\lVert}{\rVert} 

\title[CscK metrics and semistable bundles]{Constant scalar curvature K\"ahler metrics and semistable vector bundles}

\author[Annamaria Ortu]{Annamaria Ortu}
\address{Annamaria Ortu, Department of Mathematical Sciences,
Chalmers University of Technology and University of Gothenburg,
SE-412 96 Gothenburg, Sweden}
\email{ortu@chalmers.se}

\author[Lars Martin Sektnan]{Lars Martin Sektnan}
\address{Lars Martin Sektnan, Department of Mathematical Sciences,
Chalmers University of Technology and University of Gothenburg,
SE-412 96 Gothenburg, Sweden}
\email{sektnan@chalmers.se}

\begin{document}

\begin{abstract}
We give a necessary and sufficient condition for the projectivisation of a slope semistable vector bundle to admit constant scalar curvature K\"ahler (cscK) metrics in adiabatic classes, when the base admits a constant scalar curvature metric.
More precisely, we introduce a stability condition on vector bundles, which we call \emph{adiabatic slope stability}, which is a weaker version of K-stability and involves only test configurations arising from subsheaves of the bundle.
We prove that, for a simple vector bundle with locally free graded object, adiabatic slope stability is equivalent to the existence of cscK metrics on the projectivisation, which solves a problem that has been open since work of Ross--Thomas. In particular, this shows that the existence of cscK metrics is equivalent to K-stability in this setting.
We provide a numerical criterion for the positivity of the Donaldson-Futaki invariant associated to said test configurations in terms of Chern classes of the vector bundle. This criterion is computable in practice and we present an explicit example satisfying our assumptions which comes from a vector bundle that does not admit a Hermite-Einstein metric.
\end{abstract}

\maketitle 

\section{Introduction}

One of the central goals in K\"ahler geometry is to understand when a K\"ahler manifold admits a constant scalar curvature K\"ahler (cscK) metric in a given K\"ahler class. Such a metric may or may not exist, but if it does, it gives a canonical representative of a given K\"ahler class. The Yau-Tian-Donaldson conjecture \cite{Yau_OpenProblems, Tian_KahlerEinstein, DonaldsonToric02} predicts that the existence of a cscK metric should be equivalent to the algebro-geometric notion of K-(poly)stability. The conjecture is still open in general, and even in situations where it is known to hold, it is a non-trivial problem to verify whether or not K-stability holds.

In this article, we consider this question on projective bundles in \emph{adiabatic} classes. Let $\m{E} \to B$ be a simple holomorphic vector bundle over a smooth projective variety with a Hermitian metric $h$, and let $L \to B$ be a fixed ample line bundle on the base. More generally, our results will hold when $B$ is just  K\"ahler and we replace $c_1(L)$ by a K\"ahler class $\Omega$, but we keep the line bundle notation throughout for aesthetic purposes.
Consider the projectivisation $\pi:\bb{P}(\m{E}) \to B$, with the tautological line bundle $O_{\bb{P}({\m{E}})}(-1) \to \bb{P}(\m{E})$.
The dual $H := O_{\bb{P}({\m{E}})}(1)$ is a \emph{relatively ample} line bundle over $\bb{P}(\cE)$, i.e.\ the restriction on each fibre is ample.
Then the Hermitian metric $h$ induces a Hermitian metric $h^*$ on $H$ whose curvature defines a two-form on $\bb{P}(\cE)$
\[
i F_{h^*} =: \omega
\]
in $c_1(H)$ such that the restriction of $\omega$ to each fibre $\bb{P}(\cE_b)$ is the Fubini-Study metric induced by $h$.
Thus $\omega$ is a \emph{relatively K\"ahler metric}.
For each $k \gg 0$, we can define K\"ahler classes on $\bb{P}(\cE)$ by pulling back a large multiple of the base line bundle, as
\[
c_1(H) + k\pi^*c_1(L).
\]
When considering such classes for all very large $k$, these classes are called \emph{adiabatic classes}. 

This paper is concerned with determining when $\bb{P}(\cE)$ admits constant scalar curvature K\"ahler metrics in adiabatic classes. The problem dates back to Hong \cite{hong98}, who proved that when $B$ admits a cscK metric $\omega_B$ and has discrete automorphism group, and $\cE$ is slope stable, then $\bb{P}(\cE)$ admits cscK metrics in adiabatic classes.
Through the Hitchin-Kobayashi correspondence of Donaldson \cite{Donaldson1985} and Uhlenbeck--Yau \cite{UhlenbeckYau_HYMstability}, slope stability is equivalent to the bundle being simple and admitting a Hermite-Einstein metric $h$. This is used to define the relatively Fubini-Study metric $\omega$ above, which in turn is used to produce cscK metrics on $\bb{P}(E)$, using perturbative techniques. In light of the Yau-Tian-Donaldson conjecture,  a conjectural algebro-geometric counterpart of Hong's result can be interpreted as follows: when the base is K-stable and $\m{E} \to B$ is slope stable, then $\bb{P}(E) \to B$ is K-stable. Note that one can also relax the assumption that $E$ is simple. In \cite{hong02}, Hong showed that if $E$ admits a Hermite-Einstein metric, then $\bb{P}(E)$ still admits cscK metrics in adiabatic classes provided the classical Futaki invariant vanishes. However, in practice, this rarely happens for the projectivisation of vector bundles that are not simple. In general, for vector bundles that are a direct sum of simple bundles with a Hermite-Einstein metric, one produces extremal K\"ahler metrics rather than cscK ones, see \cite{bronnle15, dervansektnan21a}.

Conversely, Ross--Thomas \cite{rossthomas2006obstruction} showed that if $\m{E} \to B$ is a strictly unstable bundle, then $\bb{P}(\m{E})$ is K-unstable in adiabatic classes. It has been a long standing problem to understand the behaviour in the strictly semistable case, which we solve here. More precisely, we consider the case when $\m{E}\to B$ is a slope semistable simple vector bundle. We introduce a notion of stability, which we call \emph{adiabatic slope stability}, that allows us to prove a result that fills the gap between the results of Hong and Ross--Thomas: in particular we give a necessary and sufficient condition for the existence of cscK metrics in adiabatic classes on projectivised vector bundles.

The main result is the following.
\begin{theorem}
\label{thm:main}
Suppose $\cE \to B$ is a sufficiently smooth, simple holomorphic vector bundle over a smooth projective variety $B$. Suppose that the group of automorphisms of $B$ that lift to $L$ is discrete and that $B$ admits a cscK metric in the K\"ahler class $c_1(L)$. Then the projectivisation $X=\P(\cE)$ admits a cscK metric in the classes $c_1(H)+kc_1(L)$ for all sufficiently large $k$ if and only if $X$ is adiabatically slope stable with respect to $L$.
\end{theorem}
We first comment on our assumptions. That $\m{E}$ is sufficiently smooth means that the graded object of $\m{E}$ is locally free, see Definition \ref{def:suffsmooth}, and it is a technical assumption that allows us to use analytic techniques to prove the result. The simplicity of $\m{E}$ and the assumption on the automorphism group of $B$ ensure that $\mathbb{P}(\m{E})$ has trivial reduced automorphism group. 

The strategy of the proof is to first solve a more general equation than the cscK equation, which we do not only on $\m{E}$, but in the whole Kuranishi neighbourhood of its graded object, and for this reason we need the graded object to be locally free. To solve this more general equation we also assume that the graded object of $\m{E}$ has a Hermite--Einstein metric and that the base metric is cscK. Note that as mentioned above, Hong's results do not in general produce cscK metrics on the projectivisation of the graded object of $\m{E}$, as the graded object is not simple. Solving the more general equation reduces the cscK equation to a finite-dimensional moment map problem for every member of the Kuranishi family. We then show that the existence is implied by stability, by using a moment map flow technique on the Kuranishi space, which is finite-dimensional but non-compact.

Thus, under these assumptions, our result completely classifies when $\bb{P}(\m{E})$ admits a cscK metric in adiabatic classes. Adiabatic slope stability is defined by checking the positivity of the Donaldson-Futaki invariant with respect to \emph{some} particular test configurations, see below. Adiabatic slope stability is thus implied by K-stability. Our result is therefore in particular a verification of the Yau-Tian-Donaldson conjecture for this class of projective bundles.
\begin{corollary}
With $X$ and $L$ as above, for all sufficiently large $k$, $X$ admits a cscK metric in  $c_1(H)+kc_1(L)$ if and only if $X$ is K-stable with respect to $H+kL$.
\end{corollary}

In fact, we show that it suffices to verify the adiabatic slope stability condition on the Kuranishi space that parametrises the deformations of $\m{E} \to B$, thus reducing the infinite-dimensional problem of checking K-stability to a finite-dimensional problem on the Kuranishi space. Together with an explicit formula \eqref{Eq:Fut_expression} for the Donaldson-Futaki invariant obtained using Legendre's localisation formula \cite{LegendreLocalisation21}, this allows us to determine if adiabatic slope stability holds, and in \S \ref{sec:concreteeg} we give an example of a strictly semistable vector bundle whose projectivisation admits cscK metrics in certain adiabatic classes. This is the first such example.
\begin{corollary}
\label{cor:eg}
There exists a vector bundle $\m{E} \to B$, where $B$ is a blowup of $\bb{P}^2$ in sufficiently many points, which is strictly slope semistable with respect to some polarisation $L \to B$, and such that $\bb{P}(\m{E})$ admits cscK metrics in $c_1(\mathcal{O}(1) \otimes L^k)$ for all $k \gg 0$.
\end{corollary}

We now explain the adiabatic slope stability condition. Given a vector bundle $\m{E} \to B$, a subsheaf $\m{S} \to B$ induces a test configuration for the projectivisation, whose central fibre is given by the projectivisation of $\m{S} \oplus \frac{\m{E}}{\m{S}}$, with ample line bundle given by $H + k\pi^*L$.
Adiabatic slope stability essentially consists in the positivity of the Donaldson-Futaki invariant for these test configurations. It is therefore a vector bundle analogue of Hattori's notion of $\mf{f}$-stability \cite{Hattori_f-stability}, though note that $\mf{f}$-stability of $\bb{\m{E}}$ is stronger than adiabatic stability of $\m{E}$. The precise relationship between the two notions is explained in Remark \ref{rem:hattori}. 
Now, one consequence of the sufficiently smooth assumption is that in order to guarantee the existence of cscK metrics, we only have to consider $\m{S}$ such that $\m{S} \oplus \frac{\m{E}}{\m{S}}$ is a vector bundle. Thus this assumption allows us to only check K-stability for test configurations that arise from vector subbundles, rather than subsheaves.
A priori, on projective bundles $\mf{f}$-stability requires one to check the positivity of the Donaldson-Futaki invariant on more fibration degenerations than those arising from subsheaves, but ultimately our results show that under the sufficiently smooth hypothesis, the two are equivalent.
We expect the same results to hold without the sufficiently smooth hypothesis, though we then do not expect that it suffices to verify adiabatic slope stability on subbundles of $\m{E}$ only.

\subsection*{Strategy of the proof}
Our technique of the proof of Theorem \ref{thm:main} is quite general; we hope that our work serves as a blueprint for solving analogous ``semistable" perturbative problems.
First, we rely on the deformation theory of vector bundles and we develop new analytic estimates to reduce the cscK equation to a moment map equation in finite-dimensional K\"ahler geometry, then we use recently introduced moment map techniques to find a zero of said moment map.
We next explain the technique in more detail.

The first step consists of setting up the problem in terms of the deformation theory of the vector bundle and the associated projective bundle: the semistable holomorphic vector bundle $\m{E} \to B$ can be viewed as a deformation of its graded object $\m{E}_0\to B$, where $\m{E}$ and $\m{E}_0$ have the same underlying smooth structure, denoted $E \to 
B$.
More precisely, the $\delbar$-operator of $\m{E}$ can be written as $\delbar_0 + \gamma$, where $\gamma \in \Omega^{0,1}(\End E)$.
Let $h$ be a Hermitian metric on $E$ such that $(h, \delbar_0)$ defines a Hermite-Einstein connection on $\m{E}_0$.
Then $h$ induces a relatively symplectic structure $\omega$ on $\bb{P}(E)$ and the $\delbar$-operators $\delbar_E$ and $\delbar_0$ induce complex structures $J$ and $J_0$ on $\bb{P}(E)$ such that $(\omega, J_0)$ is a relatively K\"ahler metric and $(\omega, J)$ is a deformation of it.

At this point, we can refer to the deformation theory of complex structures, \emph{Kuranishi theory}: the deformations of $J_0$ compatible with $\omega$ and with the projection to the base can be described by a finite-dimensional vector subspace $\m{V}$. 
Moreover, this parametrisation is \emph{equivariant} with respect to the group $G$ of automorphisms of $\m{E}_0$, which is a reductive group and thus is the complexification of a maximal compact subgroup $K$.

The equivariance property is what allows us to use the theory of moment maps, for which we rely on Dervan--Hallam's recent approach using universal families \cite{DervanHallam23}.
We consider the deformations of $\bb{P}(\m{E}_0)$ as a $K$-equivariant family $\m{U} \to B \times\m{V}$. Each fibre over $v \in \m{V}$ corresponds to a projective bundle $\bb{P}(\m{E}_v) \to B$ with K\"ahler metric $(\omega+k\omega_B, J_v)$.
Dervan--Hallam prove that the projection onto the Lie algebra of $K$ of the scalar curvature of such a metric is a moment map on $\m{V}$ with respect to the Weil--Petersson metric.

The next step, and the core analytic step in our proof, consists of deforming $(\omega+k\omega_B, J_v)$ adiabatically by adding a K\"ahler potential $\varphi_k$ to the metric such that the scalar curvature itself, rather than only its projection, lies in the Lie algebra of $K$. This is a global equation, and we do not, in fact, solve it completely. We first solve a simpler fibrewise equation so that the scalar curvature of every member of the Kuranishi family lands in a finite-dimensional vector space, and then we show that the projection from this vector space to the Lie algebra of $K$ is an isomorphism, see Lemma \ref{Lemma:injective_projection}. In particular, Lemma \ref{Lemma:injective_projection} is key to deducing that the zero of the Dervan--Hallam moment map we obtain on $\m{V}$ is precisely a solution to the cscK equation on the fibre; its proof relies on the properties of the Lie algebra of $K$, it is quite general, and we expect it should be possible to apply it to solve similar perturbation problems.
The reason why we solve the fibrewise equation instead of directly solving the global equation is that the global equation is one on a non-compact manifold, due to the non-compactness of the Kuranishi space $\cV$. At this point in the argument, we have a $k$-dependent sequence of finite-dimensional moment maps on $\cV$, whose zeroes are cscK metrics. Moreover, the leading order term in the asymptotic expansion of this moment map has zeroes precisely Hermite--Einstein metrics.

It is worth noting that up to this point the proof is unobstructed, i.e.\ it does not use the adiabatic slope stability assumption.
The fact that we are solving a more general equation means we can rely on the same type of perturbative PDE techniques that have been used in earlier adiabatic problems, starting with the works of Hong \cite{hong98} and Fine \cite{fine04}, and then built upon in various generalisations \cite{luseyyedali14, bronnle15, dervansektnan20, dervansektnan21a, Ortu_OSCdeformations}. In particular, we extend to the slope semistable case the techniques developed by Br\"onnle \cite{bronnle15}, who constructed extremal metrics on the projectivisation of certain unstable vector bundles.

The last step of the proof, which is where the \emph{adiabatic slope stability} condition comes in, consists of applying the \emph{moment map flow}, a technique which has been used in analytic approaches to Geometric Invariant Theory (GIT) \cite{GeorgoulasRobbinSalamon_GITbook,ChenSun_CalabiFlow}, and for which we mostly follow \cite{DervanMcCarthySektnan}. In order to run the moment map flow, we need that the Weil--Petersson form obtained from \cite{DervanHallam23} in fact is positive. We do this by relating the Weil--Petersson metric that arises from the moduli theory of constant scalar curvature metrics, as in Fujiki--Schumacher \cite{FujikiSchumacher1990}, with the Weil--Petersson metric that arises from the moduli theory of vector bundles. The moment map flow is the gradient flow of the norm squared of the moment map; we run the flow from the point in $\m{V}$ corresponding to the adiabatically slope stable vector bundle $\m{E}\to B$ and the limit point of the flow induces a test configuration for $\bb{P}(\m{E}) \to B$. By relating the weight of the moment map at the limit point and the Donaldson-Futaki invariant of the induced test configuration, adiabatic slope stability allows us to prove that the point to which the flow converges is in fact a zero of the moment map in the orbit of our original complex structure, i.e.\ a constant scalar curvature K\"ahler metric on $\bb{P}(\cE)$.

The first main difference with our approach compared to previous constructions is that in the classical adiabatic constructions, the constant scalar curvature metric is approximated iteratively by adding a potential to the metric $\omega + k\omega_B$ so that the metric is constant up to the prescribed order. We, however, only perturb the metric so that the scalar curvature lies in the Lie algebra of $K$, up to any prescribed order.
In previous works \cite{dervansektnan21a, Ortu_OSCdeformations}, the authors have developed optimal symplectic connections as a generalisation of Hermite-Einstein connections, which allow to construct cscK and extremal metrics on the total space in adiabatic classes. The linear analysis required to perturb the metric so that the scalar curvature lies in the Lie algebra of $K$ to any required order is similar to the linear analysis involved in these problems. The main difference is that we combine using the linearisation of the scalar curvature operator and the linearisation of the contracted curvature operator on the vector bundle to achieve this, instead of looking at a higher order asymptotic expansion of the linearisation of the scalar curvature operator. This allows us to treat the adiabatic parameter $k$ and the deformation parameter $v$ separately here.

The second main difference with the classical approach to perturbation problems and deformation theory is that traditionally, e.g.\ in the Fujiki--Donaldson proof that the scalar curvature is a moment map, when studying deformations of the K\"ahler structure one fixes either the K\"ahler form or the complex structure and varies the other within the prescribed compatibility condition.
As opposed to that strategy, the Dervan--Hallam theory allows us to change both the complex structure, by varying the base point of the universal family $\m{U} \to \m{V}$, and the K\"ahler form, by adding a potential, at the same time.

Note that in \cite{DervanMcCarthySektnan}, the moment map flow is run first, and then the solution is perturbed. This makes the perturbation analysis much more delicate, as one has to have a more precise understanding of the linearisation of the relevant operator. We solve an analytic perturbation problem first before applying the moment map flow, which simplifies the analysis. The above two new elements in the approach are what allow us to tackle the problem in this order instead.

\begin{remark}
In \cite{dervansektnan21a, Ortu_OSCdeformations}, solutions to the optimal symplectic connection equation were used in order to guarantee the existence of cscK metrics on certain fibrations. In the case we are considering, where the total space is a projective bundle, this equation reduces to the Hermite-Einstein equation \cite[Proposition 3.17]{dervansektnan21a}. Our construction, and in particular the concrete example given in \S \ref{sec:concreteexample}, shows that having a solution to the optimal symplectic connection equation is \emph{not} a necessary condition for the existence of cscK metrics in adiabatic classes.
From the algebro-geometric point of view, this is saying that the notion of fibration stability introduced by the second author and Dervan in \cite{dervansektnan19b} is not equivalent to Hattori's $\mf{f}$-stability \cite{Hattori_f-stability}, and our example of Corollary \ref{cor:eg} constitutes the first example of an $\mf{f}$-stable variety that is not fibration stable.
\end{remark}

\begin{remark}
During the preparation of this article, we learned that R\'emi Delloque \cite{delloque2024HYM} was using related ideas to study a different problem: when a sufficiently smooth vector bundle that is semistable with respect to some initial polarisation can be equipped with a Chern connection such that it satisfies the Hermite-Einstein condition, upon perturbing the polarisation. His approach also consists in reducing to a finite-dimensional problem, before using a moment map flow to obtain a solution when the bundle is perturbed into the stable region. However, one difference is that Delloque's approach does not rely on the Dervan--Hallam theory for obtaining moment maps; this makes invoking the moment map flow to relate the existence of a solution to the stability criterion more involved.
\end{remark}

\subsection*{Outline} In Section \ref{sec:background}, we recall some background on cscK metrics, K-stability and the differential geometry of vector bundles and projective bundles. We also introduce adiabatic slope stability. In Section \ref{sec:linear}, we discuss the linear theory relevant to our problem. In Section \ref{sec:reduction}, we reduce the problem of finding a cscK metric to a finite-dimensional problem, which we then solve in Section \ref{sec:findimsoln} under the necessary hypothesis of adiabatic slope stability. Finally, in Section \ref{sec:example}, we give a formula for the adiabatic slope stability criterion and use it to produce examples to which the construction applies.

\subsection*{Acknowledgments} The idea for this project first came during the summer school \textit{Recent advances in K\"ahler geometry}; we thank Masafumi Hattori for many discussions we had in that instance and since, and the organisers for inviting us. We especially thank Ruadha\'i  Dervan for discussions on this and related perturbation problems. 
We also thank Leticia Brambila-Paz, R\'emi Delloque, Julien Keller, Carlo Scarpa and Carl Tipler for many useful conversations pertaining to this work. Finally, we would like to thank the referees for their comments and suggestions.
LMS is funded by a Marie Sk\l{}odowska-Curie Individual Fellowship, funded from the European Union's Horizon 2020 research and innovation programme under grant agreement No 101028041, and a Starting Grant from the Swedish Research Council (grant 2022-04574).
AO acknowledges support from the INdAM ``National Group for Algebraic and Geometric Structures, and their Applications" (GNSAGA).
The authors would like to thank the Isaac Newton Institute for Mathematical Sciences, Cambridge, for support and hospitality during the programme ``New equivariant methods in algebraic and differential geometry'' where work on this paper was undertaken. This work was supported by EPSRC grant no EP/R014604/1.

\section{Preliminaries}\label{sec:background}
In this section, we recall some results on K\"ahler manifolds and K-stability, as well as on vector bundles, their deformation theory and their projectivisation.
We then introduce the notion of adiabatic slope stability.

\subsection{K\"ahler metrics with constant scalar curvature}
Let  $(M,L)$ be a smooth projective manifold endowed with an ample line bundle.
We consider the polarisation to be a fixed datum of the various problems we describe.
Let $\omega$ be a K\"ahler form on $M$ in the first Chern class of $L$ and let $J$ be the complex structure of $M$. We denote by $g = g(\omega,J)$ the Riemannian metric on $M$ induced by $J$ and $\omega$, i.e.\
\[
g(\cdot, \cdot) = \omega(\cdot, J\cdot).
\]
The \emph{Ricci curvature} of $\omega$ is the two-form
\[
\mrm{Ric}(\omega, J) = -\frac{i}{2\pi} \del_J \delbar_J \mrm{log} \ \omega^n.
\]
The \emph{scalar curvature} of the K\"ahler metric $(\omega,J)$ is a smooth function on $M$ defined as the contraction of the Ricci curvature:
\[
\mrm{Scal}(\omega,J) := \Lambda_{\omega} \mrm{Ric}(\omega, J).
\]

We are interested in \emph{special} K\"ahler metrics, where the scalar curvature is subject to certain constraints. In particular, we consider K\"ahler metrics with constant scalar curvature, where the constant is given by the intersection product
\[
\widehat{S} = \frac{n \ c_1(M) \cdot c_1(L)^{n-1}}{c_1(L)^n}.
\]
In particular, $\widehat{S}$ is a \emph{topological constant} fixed by the polarisation.
Fixing a reference K\"ahler metric $\omega \in c_1(L)$, we can then describe the linearisation of $\mrm{Scal}(\omega)$ in the direction of a K\"ahler potential $\varphi$.

\begin{definition}
Let $\m{D} : C^\infty(M, \bb{C}) \to \Omega^{0,1}(T^{1,0}M)$ be the operator
\[
\m{D}(\varphi) = \delbar (\nabla^{1,0}_g \varphi).
\]
The Lichnerowicz operator is the composition $\m{D}^*\m{D}$, where $\m{D}^*$ is the adjoint defined with respect to the $L^2(g)$-inner product.
\end{definition}
It can be written explicitly as follows:
\[
\m{D}^*\m{D}(\varphi) = \Delta^2_g(\varphi) + \langle \mrm{Ric}(\omega), i \del\delbar\varphi\rangle +\frac{1}{2} \langle \nabla \mrm{Scal}(\omega), \nabla \varphi\rangle.
\]
The Lichnerowicz operator is a 4th-order elliptic operator. Its kernel, which by compactness is the kernel of $\m{D}$, coincides with the space of real holomorphy potentials on $M$. 
The linearisation of the scalar curvature in the direction of a K\"ahler potential $\varphi$ can be written in terms of the Lichnerowicz operator as
\begin{equation}\label{eq:linearisation_scalar_curvature}
-\m{D}^*\m{D}(\varphi) + \frac{1}{2}\langle \nabla \mrm{Scal}(\omega), \nabla \varphi \rangle.
\end{equation}
In particular, the linearisation at a constant scalar curvature metric is given exactly by the Lichnerowicz operator.

\subsection{K-stability}
We recall the notion of K-stability, introduced by Tian \cite{Tian_KahlerEinstein} and Donaldson \cite{DonaldsonToric02} to be an analogue for polarised varieties to the Hilbert-Mumford criterion for GIT.
\begin{definition}
Let $(X,L)$ be a polarised variety of dimension $n$. A test configuration for $(X,L)$ is the data of
\begin{enumerate}
\item a variety $\m{X}$ with a $\bb{C}^*$-equivariant flat proper morphism to $\bb{C}$;
\item a relatively ample line bundle $\m{L} \to \m{X}$ together with a lift of the $\bb{C}^*$-action to it;
\item an isomorphism $(\m{X}_1, \m{L}_1) \simeq (X, rL)$ for some $r>0$.
\end{enumerate}
We say that $(\m{X},\m{L})$ is a product test configuration if $(\m{X}, \m{L}) \simeq (X,rL) \times \bb{C}$, with a possibly nontrivial $\bb{C}^*$-action, and is a trivial test configuration if $(\m{X}, \m{L}) \simeq (X,L) \times \bb{C}$ with trivial $\bb{C}^*$-action.
\end{definition}
Given a test configuration, one associates a numerical invariant as follows. Consider the following expansions for the dimension of $H^0(\m{X}_0, j\m{L}_0)$ and for the weight $w_j$ of the induced action of $\bb{C}^*$ on $H^0(\m{X}_0, j\m{L}_0)$:
\[
\begin{aligned}
\dim H^0(\m{X}_0, j\m{L}_0) &= a_0j^n + a_{1}j^{n-1} + O\left(j^{n-2}\right)\\
w_j &= b_0 j^{n+1} + b_1j^{n} + O\left(j^{n-2}\right).
\end{aligned}
\]
The \emph{Donaldson-Futaki} invariant, introduced by Donaldson in \cite{DonaldsonToric02}, is the number
\[
\mrm{DF}(\m{X}, \m{L}) := \frac{a_1b_0-a_0b_1}{a_0^2}.
\]
It is related to the classical Futaki invariant by the following result, first proved by Donaldson in the projective case \cite[Proposition 2.2.2]{DonaldsonToric02} and extended by Legendre \cite[Theorem 1.1]{LegendreLocalisation21}. 
\begin{theorem}\label{Thm:classical_Futaki_equal_DF}
Assume that the central fibre $\m{X}_0$ is smooth. Let $\Fut(J\xi)$ be the classical Futaki invariant
\[
\Fut(J\xi) = \int_{\m{X}_0} f (\Scal(\omega)-\widehat{S}) \omega^n
\]
where $\xi = \mrm{grad}^\omega f$ is the Hamiltonian vector field generating the $\bb{C}^*$-action, i.e.\ the symplectic gradient of $f$.
Then the Donaldson-Futaki invariant of the test configuration $(\m{X}, \m{L})$ satisfies
\[
\frac{\mrm{DF}(\m{X}, \m{L})}{n!} = -\pi\Fut(J\xi)
\]
\end{theorem}
If $(\m{X}, \m{L})$ is a test configuration, one can define the \emph{normalisation} of $(\m{X}, \m{L})$ by taking the normalisation $\widetilde{\m{X}}$ of $\m{X}$ and the pullback of $\m{L}$ to $\widetilde{\m{X}}$. The normalisation is again a test configuration for $(X, L)$ \cite[\S 5]{RossThomas07}.
\begin{definition}[\cite{DonaldsonToric02,Tian_KahlerEinstein}]
A polarised variety $(X, L)$ is
\begin{enumerate}
\item \emph{K-semistable} if $\mrm{DF}(\m{X}, \m{L}) \ge 0$ for all test configurations $(\m{X}, \m{L})$ for $(X, L)$;
\item \emph{K-polystable} if it is K-semistable and $\mrm{DF}(\m{X}, \m{L})$ vanishes only if $(\m{X}, \m{L})$ normalises to a product test configuration;
\item \emph{K-stable} if it is K-semistable and $\mrm{DF}(\m{X}, \m{L})$ vanishes only if $(\m{X}, \m{L})$ normalises to the trivial test configuration.
\end{enumerate}
\end{definition}

\subsection{Differential geometry of vector bundles}\label{Subsec:diffgeomvectorbundles}

Let $\m{E} \to B$ be a holomorphic vector bundle of rank $r+1$. Assume that $\m{E}\to B$ is strictly slope semistable with respect to an ample line bundle $L \to B$, and let $\omega_B \in c_1(L)$ be a K\"ahler metric. Let 
$$
0=\m{S}_0 \subset \m{S}_1 \subset \m{S}_2 \subset \ldots \subset \m{S}_m = \m{E}
$$
be a Jordan--H\"older filtration for $\m{E}$. Let $\m{Q}_i = \frac{\m{S}_i}{\m{S}_{i-1}}$ be the successive quotients in this Jordan--H\"older filtration. In general, the $\m{Q}_i$ are torsion-free sheaves, but we will make the following assumption on  $\cE$, which is a definition that goes back to work of Leung.
\begin{definition}[{\cite{leung97}}]A vector bundle $\m{E} \to B$ which is strictly slope semistable with respect to $L$ is \emph{sufficiently smooth} if there is a Jordan--H\"older filtration for $\m{E}$ such that the successive quotients $\m{Q}_i$ are all locally free. 
\end{definition}\label{def:suffsmooth}
Hence the $\m{Q}_i$ are slope stable vector bundles with the same slope as $\m{E}$. The graded object
\[
\m{E}_0 = \bigoplus_{i=1}^m \m{Q}_i
\]
of $\m{E}$ is then locally free and slope polystable with respect to $L$. Note that the Jordan--H\"older filtration of $\m{E}$ is not unique, but that its graded object is, and so the notion is in fact independent of choice of Jordan--H\"older filtration. In particular, in this case, $\m{E}_0$ and $\m{E}$ are isomorphic as smooth vector bundles, with different $\db$-operators: let $\db_0$ be the $\db$-operator of $\m{E}_0$ and $\db$ that of $\m{E}$.
Letting $E$ denote the underlying smooth vector bundle, we may write 
\begin{align*}
\cE_0 =& (E, \db_0) \\
\cE =& (E, \db).
\end{align*}
Analogously, we denote the underlying smooth bundles throughout by $S_j$ and $Q_j$.
Being slope polystable, by the Hitchin--Kobayashi correspondence of Donaldson, Uhlenbeck and Yau \cite{Donaldson1985, UhlenbeckYau_HYMstability}, $\m{E}_0$ admits a Hermite-Einstein metric $h$, i.e.\
\[
\Lambda_{\omega_B} F_{\nabla_0} = c \Id,
\] 
where $\nabla_0$ is the Chern connection induced by $h$ and $\delbar_0$ on $E$ and $c$ is a constant given by the slope of $\m{E}_0$ divided by the volume of $\omega_B$.

The linearisation of the Hermite-Einstein operator will play an important role in our work, and this is given by the following result.
\begin{lemma}\label{Lemma:linearisationHE}
The linearisation of the Hermite-Einstein operator of a Chern connection $\nabla$ when changing the $\delbar$-operator is given by the corresponding Laplacian operator $\Delta$ on endomorphisms of $\m{E}$.
\end{lemma}

The space $\cJ_E$ of holomorphic structures on $E$ compatible with $h$ is an affine space modelled on $\Omega^{0,1}(\End E)$, so $\db_0$ and $\db$ are related by
$$
\db = \db_0 + \gamma
$$
for some $\gamma \in \Omega^{0,1}(\End E)$.
We can decompose $\gamma = \sum_{i,j} \gamma_{i,j}$ with $\gamma_{i,j} \in \Omega^{0,1} (Q_j^* \otimes Q_i)$. Moreover, since the $\m{S}_j$ are holomorphic subbundles of $\cE$ and the holomorphic structure of each $\m{Q}_j$ is the holomorphic structure as a quotient $\frac{\m{S}_j}{\m{S}_{j-1}}$, all $\gamma_{i,j}$ with $i \geq j$ vanish. Thus the matrix $\gamma = [\gamma_{i,j}]$ is strictly upper-triangular.

\subsubsection{Deformation theory}\label{subsubsec:kuranishi_vb}
Let $\m{J}_E$ be the space of pseudo-holomorphic structures on $E \to B$. 
By taking a slice of the action of the complex gauge group on $\m{J}_E$ at $\delbar_0$, we consider the deformation space of the holomorphic structure $\delbar_0$
\[
H^1(B, \End \cE_0)
\]
which can be identified with the finite dimensional subspace of $\Omega^{0,1}(\End E)$ consisting of harmonic representatives. In particular, these are first-order integrable deformations.

Let $G = \Aut\cE_0$ be the group of holomorphic bundle automorphisms of $\cE_0$, and let $K$ be the subgroup of unitary such automorphisms. Note that $G$ is the complexification of $K$.
The Lie algebra $\mathfrak{g}$ of $G$ is the kernel of the Laplacian on $\End \cE_0$, and the Lie algebra $\mathfrak{k}$ of $K$ is the restriction of the kernel to endomorphisms that are skew-Hermitian with respect to $h$.
In particular, the group $G$ is a subgroup of the complex gauge group, which acts on the space of $\db$-operators via conjugation. If $\db$ is $\db$-operator and $f \in \GL(E)$ is a complex gauge transformation, this is given by
$$
f \cdot \db = f^{-1} \circ \db \circ f.
$$
From the deformation theory of vector bundles \cite[\S 7.2, 7.3]{Kobayashi_book}, there exists a neighbourhood $\m{V}$ of the origin in $H^1(B, \End \cE_0)$ and a holomorphic injective map
\begin{equation}\label{eq:Kuranishimap_vb}
\widetilde \Phi : \cV \to \cJ_E,
\end{equation}
such that $\widetilde \Phi (0) = \db_0$.
The group $G$ admits a local action on the space $\cV$ making the map $\widetilde \Phi : \cV \to \cJ_E$ $G$-equivariant.
Since we are including also non-integrable $\delbar$-operators, the space $\m{V}$ is a ball in $H^1(B, \End \cE_0)$; we denote by $\m{V}^{\mrm{int}}$ the complex subspace of $\m{V}$ parametrising integrable complex structures.

The holomorphic operator $\delbar$ of $\m{E}\to B$ is then a deformation of $\delbar_0$ such that it is represented by a point $v \in \m{V}$ for which there exists a one-parameter subgroup $\bb{C}^*$ of $K$ with $0 \in \overline{\bb{C}^*\cdot v}$.
Through the map \eqref{eq:Kuranishimap_vb}, this orbit defines a family $\m{F} \to B \times \m{V}$ that represents the degeneration of $\m{E} \to B$ to the graded object.
In particular, we have a family of $\delbar$-operators
\[
\delbar_{v} = \delbar_0 + \gamma_v,
\]
where $v$ runs in a $\bb{C}^*$-orbit whose closure contains the origin.
Moreover, $h$ is also a Hermitian metric on $\m{E}$ and the degeneration of $\m{E}\to B$ to $\m{E}_0 \to B$ induces a family of Chern connections $\nabla_v = \nabla(h, \delbar_v)$.
The curvature of the connection $\nabla_v$ is related to that of $\nabla_0$ as follows \cite[\S4.3]{huybrechts05}
\begin{equation}\label{Eq:expansion_curvature_vb}
F_v = F_0 + \dd_0(\gamma_v - \gamma_v^*) + (\gamma_v - \gamma_v^*)\wedge (\gamma_v - \gamma_v^*).
\end{equation} 

\subsection{Projectivised vector bundles}
\label{sec:projvb}
Let $\pi:X = \bb{P}(\m{E}) \to B$ be the projectivisation of $\m{E} \to B$.
Let $H = \m{O}_{\bb{P}(\m{E})}(-1)^*$ be the relatively ample line bundle over $X$ such that its restriction to each fibre $X_b$ of $\pi$ is $\m{O}_{\bb{P}(\m{E}_b)}(1)$. We also fix an ample line bundle $L$ over the base.

For each $k \gg 0$, an ample line bundle over $X$ is given by
\[
H + k\pi^*L.
\]
In the following, we will often omit the pullback from the notation and write $H + kL$ with a slight abuse of notation.

The $\delbar$-operator on $\m{E} \to B$ induces an almost complex structure $J$ on the underlying smooth manifold $M$ of $X$. Moreover, the fixed Hermitian metric $h$ on $E$ induces a Hermitian metric on $\mathcal{O}(1)$; the curvature of the Chern connection obtained by $h$ and $\delbar$ gives rise to a relatively K\"ahler metric
\[
(\omega = i F_\nabla, J)
\]
on $M \to B$.
Moreover, the restriction of $(\omega, J)$ on each fibre of $\pi$ is a Fubini-Study metric on the fibre.
For each $k \gg 0$, a K\"ahler metric on $X$ is given by
\[
\omega_k = \omega+ k\omega_B
\]
in the \emph{adiabatic class} $c_1(H) + kc_1(L)$. The volume form of $\omega_k$ expands as
\begin{align}\label{eq:adiabaticvolform}
\omega_k^{n+r} = {n+r \choose n} k^n \omega^r \wedge \omega_B^n + O(k^{n-1}).
\end{align}

\subsubsection{Kuranishi theory}\label{subsubsec:kuranishi_proj}
The fibration $X \to B$ degenerates to the projectivisation $\bb{P}(\m{E}_0) \to B$ through the degeneration family $\varpi:\m{X} = \bb{P}(\m{F}) \to B\times \m{V}$.
We denote the central fibration of the family $\varpi$ by $\pi_0:\m{X}_0 = \bb{P}(\m{E}_0) \to B$.

More precisely, let $\cJ_\pi$ denote the space of almost complex structures on $M$, compatible with $\omega$ and which preserve the projection $\pi$ to $B$. There is a map $\iota : \cJ_E \to \cJ_\pi$, and we thus obtain a map
\begin{equation}\label{eq:Kuranishimap_proj}
\Phi : \cV \to \cJ_\pi,
\end{equation}
by composing $\widetilde{\Phi}$ \eqref{eq:Kuranishimap_vb} with $\iota$. Moreover, the map $\Phi$ is equivariant with respect to $G$.
We refer to $\m{V}$ as the \emph{Kuranishi space} about $\m{E}_0 \to B$ and to $\widetilde\Phi$ as the \emph{Kuranishi map}, as they are obtained from a fibrewise version of the deformation theory of complex structures of Kuranishi \cite{Kuranishi_family_cpx_str}, \cite[Theorem 4.7]{Ortu_OSCdeformations}.

There is a universal family over $\cJ_\pi$, which, as a smooth manifold, is simply $\cJ_\pi\times M$. Pulling back the universal family via $\Phi$ gives a universal family
$$
\pi_{\m{U}}:\cU \to \cV.
$$
The family $\m{U}$ is diffeomorphic to $\m{V}\times M$, thus it also admits a projection
\[
\m{U} \to B \times \m{V}
\]
and a projection
\[
\mrm{pr}_2 : \m{U} \to M.
\]
There is then an induced action of $G$ on the image of $\cJ_E$ via $\iota$ and so on $\cU$, making $\cU \to \cV$ a $G$-equivariant holomorphic map.

Finally, after potentially shrinking $\cV$, we can assume that $\omega_k$ is non-degenerate and compatible with the holomorphic structure $\Phi(v)$ for every $v \in \cV$ for all sufficiently large $k$.
Therefore, we have a family of almost K\"ahler structures
\[
(\omega_k, J_v)
\]
on each $X_v \to B$.
Moreover, $\omega_k$ induces a relatively almost K\"ahler metric, which we denote $\omega_{\m{U},k}$, on the universal family $\cU$, for all sufficiently large $k$. This is K\"ahler when restricted to the integrable fibres.

\subsubsection{Splitting of the function space}
The relatively K\"ahler form $\omega$ on $X$ induces a splitting of the function space $C^\infty(X, \bb{R})$ as
\[
C^\infty(X) = C^\infty(B) \oplus C^\infty_0(X),
\]
where the $C^\infty(B)$-part is obtained taking the fibre integral
\[
f \mapsto \int_{X/B}f \omega^r
\]
and $C^\infty_0(X)$ denotes the fibrewise mean-value zero functions.
Furthermore, we can split
\[
C^\infty_0(X) = C^\infty_W(X) \oplus C^\infty_R(X),
\]
where $C^\infty_W(X)$ denotes the space of \emph{fibrewise holomorphy potentials} on $\m{X}_0 \to B$, which coincides with the kernel of the vertical Lichnerowicz operator $\m{D}_{V}^*\m{D}_V$. For a general holomorphic submersion, these can be identified with the smooth sections of a finite rank vector bundle $W \to B$, with the identification depending on the relatively K\"ahler metric on $X$.
The space $C^\infty_R(X)$ is given by the $L^2(\omega)$-orthogonal complement to $C^{\infty}_W(X)$ in $C^\infty_0(X)$.

We next describe the space $C^\infty_W(X)$ of fibrewise holomorphy potentials, following \cite[\S2]{luseyyedali14}.
Suppose $h$ is a Hermitian metric on $E$, which in our case we take to be the Hermite-Einstein metric for $\cE_0$. For every $x\in E_b \setminus \{ 0 \}$, we then get an endomorphism $\lambda_x$ of the fibre $E_b$ of $E \to B$ via
$$
\lambda_x (-) = \frac{x \otimes h(-, x)}{\|x\|^2_h}.
$$ 
Note that $\lambda$ is scale-invariant. For every $\sigma \in \End(E)$, the map
$$
x \mapsto - \tr (\lambda_x \circ \sigma)
$$ 
gives rise to a function on $E_b \setminus \{ 0 \}$, which moreover is scale-invariant. Thus this descends to a function $h_\sigma$ on $\P(E_b)$, defining a map
\[
\begin{aligned}
m : \Gamma(\End E) &\to C^{\infty}(X)\\
\sigma &\mapsto h_\sigma.
\end{aligned}
\]
Note that $m(\Id_E) = -1$, so, since $m$ is $C^{\infty}(B)$-linear, $m(f \cdot \Id_E) = -f$ for all $f \in C^{\infty}(B)$. When restricting to the space $\End_0 E$ of fibrewise trace-free endomorphisms of $E$, the functions $h_{\sigma}$ are then of fibrewise average $0$. Moreover, since the induced metric on the fibres is a Fubini-Study metric, one can, from the formula for the moment map for the action of $U(r+1)$ on $\mathbb{P}^r$, show that the $h_{\sigma}$ are in fact fibrewise holomorphy potentials on $\bb{P}(\cE_v)$ for any $v \in \cV$.

Note also that the map $m: \Gamma(\End_0 E) \to C^{\infty}_W$ is a map of Lie algebras. First, since trace and (post)-composition are linear maps, the map $m$ is linear. Next, for $\sigma \in \Gamma(\End_0 E)$, let $\xi_{\sigma}$ denote the corresponding vector field on $\P(E)$, which then satisfies $-dh_{\sigma} = \iota_{\xi_{\sigma}} \omega,$ where $\omega$ is the fibrewise Fubini-Study metric induced by the Hermitian metric $h$ on $E$. We then have
\begin{align*}
dh_{[\sigma, \tau]} =& -\iota_{\xi_{[\sigma, \tau]}} \omega \\
=& - \iota_{[\xi_{\sigma},\xi_{\tau}]} \omega \\
=& d \{ h_{\sigma}, h_{\tau} \},
\end{align*}
where in the second equality we use that the infinitesimal action $\sigma \mapsto \xi_\sigma$ is a Lie algebra homomorphism.
Here $\{ h_{\sigma}, h_{\tau} \}$ denotes the Poisson bracket of $h_{\sigma}$ and $h_{\tau}$, which is defined as the function $h$ such that $dh = \iota_{[\xi_{\sigma},\xi_{\tau}]} \omega$. This does not necessarily define $h$ uniquely as a function on $X$, since $\omega$ is only relatively K\"ahler, but it still defines $h$ uniquely when requiring $h$ to be in $C^{\infty}_W$.

Summing up, we have the following result.
\begin{lemma}\label{Lemma:CinftyE_correspondence}
The map
\begin{equation}\label{Eq:CinftyE_correspondence}
m : \Gamma(\End_0 E) \to C^{\infty}_W,
\end{equation}
is an isomorphism of Lie algebras, where the Lie algebra structure on $C^{\infty}_W$ is given by the Poisson bracket.
\end{lemma}

\subsubsection{The Lie algebra $\mf{k}$}
\label{sec:mfk}
We consider the Lie algebra $\mfk$ of $K$ as a subspace of $\Gamma(\End E)$. Thus we can view $\mfk$ as a subset of $C^{\infty}_W \subset C^{\infty}(M)$. We end the section by introducing some convenient notations for various incarnations of $\mfk$ in our setup, building on this. 
Since $K$ acts on $\cU, \cV$ and $M$, elements of $\mfk$ induce vector fields on each of these manifolds. For $\sigma \in \mfk$, let
\begin{enumerate}
\item $\nu_{\sigma} \in \Gamma(T\cU)$ denote the corresponding vector field on $\cU$:
\begin{equation}\label{eq:infinitesimalvectorfield}
\nu_\sigma(x) = \left.\frac{\dd}{\dd t}\right\vert_{t=0} \exp(-i\sigma)\cdot x;
\end{equation}
\item $\eta_{\sigma} \in \Gamma(T\cV)$ denote the corresponding vector field on $\cV$;
\item $\xi_{\sigma} \in \Gamma(TM)$ denote the corresponding vector field on $M$.
\end{enumerate}
Since $\cU=\cV \times M$ as a smooth manifold, we can pull back vector fields to $\cU$ through the projections $\pr_1:\cU \to \cV $ and $\pr_2:\cU \to \cV$. Moreover, as the $K$-action is a product action, we get that $\nu_{\sigma} = \pr_1^* \eta_{\sigma} + \pr_2^* \xi_{\sigma}$, which is the horizontal-vertical decomposition of $\nu_{\sigma}$ with respect to the fibration structure $\cU \to \cV$.

Let 
$$
\overline \mfk^M = m (\mfk) = \{ \widetilde{h}_{\sigma} : \sigma \in \mfk \},
$$
where $\widetilde{h}_{\sigma} = m (\sigma) \in C^{\infty}(M)$. Then $\widetilde{h}_{\sigma}$ is a Hamiltonian for $\xi_{\sigma}$ on $M$ with respect to the fibrewise Fubini-Study metric induced by a Hermitian metric on $E$.
If $\psi \in C^{\infty}(M)$ is a K\"ahler potential, we let 
$$
\overline \mfk^M_{\psi} = \{ \widetilde{h}_{\sigma} + \xi_{\sigma} (\psi) : \sigma \in \mfk \},
$$
which are Hamiltonians on $M$ of $\xi_{\sigma}$ with respect to $\omega_k + \ddb \psi$, which moreover are holomorphy potentials with respect to the complex structure on $M$ corresponding to the central fibre of the family $\cV$. 

Let 
$$
h_{\sigma} = \pr_2^* (\widetilde{h}_{\sigma}) \in C^{\infty}(\cU)
$$
and, for $\phi \in C^{\infty}(\cU)$, let
\begin{equation}\label{eq:changeinholpot}
{h}_{\sigma}^\phi = h_{\sigma} + \nu_{\sigma} (\phi).
\end{equation}
We have already shown that $h_{\sigma}$ is a Hamiltonian for $\nu_{\sigma}$ with respect to the relatively almost K\"ahler metrics $\omega_{\m{U},k}$ on $\cU$ for any $k$.
We let
$$
\overline \mfk = \{ h_{\sigma} : \sigma \in \mfk \}
$$
and
$$
\overline \mfk_{\phi} = \{ {h}_{\sigma}^\phi : \sigma \in \mfk \}.
$$

\begin{lemma}
\label{lem:changeinholpot}
Let $\sigma \in \mfk$ and let $h_{\sigma}$ be the corresponding Hamiltonian with respect to $\omega_{\m{U},k}$ on $\cU$.
Over the integrable locus $\cV^{\mrm{int}}$,
$
{h}_{\sigma}^\phi
$
is a Hamiltonian for $\sigma$ with respect to $\omega_{\m{U},k}+ \ddb \phi$.
\end{lemma}
\begin{proof}
Fibrewise, this is the classical formula for the change in holomorphy potential, see e.g. \cite[Lemma 4.10]{szekelyhidi14book}. Since $\nu_{\sigma}$ is a vertical holomorphic vector field, this gives the global formula in this case too, even though $\omega_{\m{U},k}$ and $\omega_{\m{U},k} + \ddb \phi$ are just relatively K\"ahler.
\end{proof}
\begin{remark}
For a section $\sigma \in \mf{k}$, the Hamiltonian $h_\sigma$ is a global holomorphy potential on $\bb{P}(\m{E}_0) \to B$, i.e.\ with respect to the complex structure $J_0$. As remarked above,  $h_\sigma$ is also a fibrewise holomorphy potential for the complex structure $J_v$ for every $v \in \m{V}^{\mrm{int}}$, but may not be globally holomorphic. For example, for $\bb{P}(\cE)$, none of the non-trivial vector fields obtained in this way are globally holomorphic, since $\cE$ is simple.
\end{remark}


\subsection{Adiabatic slope stability}
\label{sec:adiabaticstability}
Let $\m{E}$ be a vector bundle and $\m{S} \subset \cE$ be a reflexive subsheaf. We can form a family of sheaves
$$
\widetilde{\cE} \to B \times \C
$$ 
such that the fibres $\widetilde{\cE}_t \cong \cE$ for $t \neq 0$ and the central fibre $\widetilde{\cE}_0$ coincides with the graded object $\m{S} \oplus \frac{\cE}{\m{S}}.$ Taking the projectivisation, we therefore get a family
\begin{align*}
(\cX_{\m{S}}, \cH) \to B \times \C
\end{align*}
where $\cH$ is the $\mathcal{O}(1)$-bundle on $\cX_{\m{S}} = \P(\widetilde{\cE})$.
Since $\cX_{\m{S}}$ has a map to $B$, we can pull back $L$ to $\cX_{\m{S}}$, so that $\m{H}+kL$ is a relatively ample line bundle on $\m{X}_{\m{S}}\to \bb{C}$.
By projecting onto $\bb{C}$, we obtain a test configuration
\[
(\cX_{\m{S}}, \cH+kL) \to \C,
\]
where the $\bb{C}^*$-action is the trivial action on $\bb{C}$ pulled-back on $(\cX_{\m{S}}, \cH+kL)$.
This leads us to the following stability notion for a vector bundle, through the Donaldson-Futaki invariant of such test configurations.
\begin{definition}
A vector bundle $\cE$ is said to be \emph{adiabatically slope stable with respect to $L$} if for all reflexive subsheaves $\m{S}$ of $\cE$,
\begin{align*}
\mrm{DF}(\cX_{\m{S}}, \cH + kL) > 0
\end{align*}
for all $k \gg 0$.
\end{definition}
Since $\m{S}$ is a subsheaf of a vector bundle over a smooth base, the assumption that $\m{S}$ is reflexive is equivalent to it being a saturated subsheaf of $\m{E}$, i.e.\ the quotient $\frac{\m{E}}{\m{S}}$ being torsion-free \cite[Proposition 1.1]{Hartshorne1980}.
Thus, following \cite[\S5.4.3]{rossthomas2006obstruction}, we can expand the Donaldson-Futaki invariant in powers of $k$ and write
\[
\mrm{DF}(\cX_{\m{S}}, \cH + kL) = W_0(\m{S}) k^n + W_1(\m{S}) k^{n-1} + \Ok{n-2}
\]
So $\m{E}$ is adiabatically slope stable if for all $\m{S}$ there exists an order $i$ such that
\begin{equation}\label{Eq:adiabatic_slope_stability_f_stability}
W_0(\m{S}) = \dots = W_i(\m{S}) = 0 \ \text{and} \ W_{i+1}(\m{S}) >0.
\end{equation}

\begin{remark}\label{rem:hattori}
With this formulation, adiabatic slope stability can be viewed as a vector bundle analogue to Hattori's $\mf{f}$-stability \cite{Hattori_f-stability}.
In fact, a proper holomorphic submersion $(X, H) \to (B,L)$ is $\mf{f}$-stable if the condition \eqref{Eq:adiabatic_slope_stability_f_stability} holds for all non-trivial fibration degenerations. Thus the difference between the two notions is that in $\mf{f}$-stability one requires this numerical condition to hold on a wider class of degenerations than just the $\mathcal{X}_{\m{S}}$.
\end{remark}

\begin{remark}\label{Rmk:explicit_testconf}
In the case that $\m{S} \subset \m{E}$ is a vector subbundle, the test configuration $\m{X}_{\m{S}}$ can be described explicitly using the language of Kuranishi theory introduced in \S\ref{subsubsec:kuranishi_proj}.
Since this is the case of interest when $\m{E}$ is sufficiently smooth, as we assume in the hypotheses of Theorem \ref{thm:main}, we now describe this test configuration in more detail.
We consider the deformation space $\m{V}$ parametrising the deformations of $\m{S} \oplus \frac{\m{E}}{\m{S}}$.
Let $v \in \m{V}$ be the point corresponding to $\delbar_{\m{E}}$ via the map $\widetilde \Phi$ \eqref{eq:Kuranishimap_vb} and let $\sigma \in \mf{k}$ be the element in the Lie algebra such that
\[
\lim_{t \to 0} \exp(-it\eta_\sigma)\cdot v =0,
\]
where $\eta_{\sigma}$ is the induced vector field on $\m{V}$.

Composing the Kuranishi map $\Phi$ \eqref{eq:Kuranishimap_proj} together with such a one-parameter subgroup, we obtain an $S^1$-equivariant map
\[
F : \Delta \to \scr{J}.
\]
We define the total space $\m{X}_{\eta_\sigma}$ of the test configuration to be $\bb{P}(\m{E}_v) \times \Delta$ as a smooth manifold, where the complex structure on each fibre over $t \in \Delta$ is given by $F(t)$ and it is integrable because $F$ is holomorphic. The $S^1$-action on $\m{X}_{\eta_\sigma}$ is given by
\begin{equation}\label{Eq:action_induced_testconfig}
s \cdot (x, t) = (\exp(-is \eta_\sigma)\cdot x, s t).
\end{equation}

In terms of $\db$-operators, the operator $\db_{\m{E}}$ on $\m{E}$ corresponding to the point $v$ differs from the operator $\db_0$ on $\m{S} \oplus \m{E}/\m{S}$ by an element $\gamma \in \Omega^{0,1}((E/S)^* \otimes S)$ and the family of $\db$-operators giving the complex structure on the fibre over $t \in \C$ is given by
$$
\db_t = \db_0 + t \gamma.
$$
The action $\ast$ of $S^1$ is then 
\begin{align*}
e^{i\theta} \ast \db_t =& \db_{e^{i \theta}t} \\
=& e^{-i \theta} \cdot \db_t \\
=& \exp(i \theta \Id_S) \circ \db_t \circ \exp(-i \theta \Id_S)
\end{align*}
in terms of the conjugation action. In particular, the vector field corresponding to the push-forward of the angular vector field $\partial_{\theta}$ on the total space of the test configuration is then the one induced by $-i \Id_S$.

The line bundle on $\m{X}_{\eta_\sigma}$ is given by the pull-back under the projection $\bb{P}(\m{E}_v) \times \Delta \to \bb{P}(\m{E}_v)$ of the line bundle $\m{O}_{\bb{P}(\m{E}_v)}(-1)^*+kL$, and the $S^1$-action \eqref{Eq:action_induced_testconfig} lifts to a holomorphic $S^1$-action on this line bundle.
\end{remark}

The test configuration that we obtain in this way has the same Donaldson-Futaki invariant as the deformation to the normal cone of $\bb{P}(\m{E})$ along $\bb{P}(\m{S})$ with polarisation parameter equal to the Seshadri constant of $\bb{P}(\m{S})$, as explained in \cite[Remark 5.14]{rossthomas2006obstruction}.
In fact, given a reflexive subsheaf $\m{S} \subset \m{E}$, the deformation to the normal cone is a test configuration defined as the blow-up along $\bb{P}(\m{S}) \times \{0\}$ of $\bb{P}(\m{E}) \times \bb{P}^1$, and the central fibre is given by $\mrm{Bl}_{\bb{P}(\m{S})}\bb{P}(\m{E})$ glued to the exceptional divisor of $\mrm{Bl}_{\bb{P}(\m{S}) \times \{0\}} \bb{P}(\m{E}) \times \bb{P}^1$ along the exceptional divisor of $\mrm{Bl}_{\bb{P}(\m{S})}\bb{P}(\m{E})$.
The central fibre of our test configuration can be obtained from the deformation to the normal cone after the contraction of the component $\mrm{Bl}_{\bb{P}(\m{S})}\bb{P}(\m{E})$ in the central fibre.

The Donaldson-Futaki invariant of this test configuration can then be related to the slope through the following.
\begin{proposition}[{\cite[\S5.4.3]{rossthomas2006obstruction}}] 
\label{prop:RossThomas}
The Donaldson-Futaki invariant admits an expansion
\begin{equation}\label{eq:DFexpansion_RossThomas}
\mrm{DF}(\cX_{\m{S}}, \cH + kL) = C k^{n-1} (\mu_L(\cE)-\mu_L(\m{S})) + O(k^{n-2}).
\end{equation}
\end{proposition}

The relationship between slope stability and adiabatic slope stability is similar to the relationship between slope stability and Gieseker stability.
\begin{lemma}
Slope stability implies adiabatic slope stability, and adiabatic slope semistability implies slope semistability.
\end{lemma}
\begin{proof}
Both statements follow from the Ross--Thomas expansion \eqref{eq:DFexpansion_RossThomas} of the Donaldson-Futaki invariant.
\end{proof}

\begin{remark}
In this work, we restrict to only consider subbundles of $\cE$. This leads to an \emph{a-priori} weaker notion of adiabatic slope stability with respect to subbundles, where we only require $\mrm{DF}(\cX_{\m{S}}, \cH + kL) > 0$ when $\m{S}$ is a subbundle of $\cE$. Under our sufficiently smooth hypothesis on $\cE$, this will turn out to be equivalent to adiabatic slope stability, see Corollary \ref{cor:adiabaticsubbundles}. In general, when $\cE$ is not sufficiently smooth, we do not expect this equivalence.
\end{remark}

\section{The behaviour of the linearised operator under extensions}
\label{sec:linear}

In this section, we discuss the linear theory of the problem. As our construction relies on the implicit function theorem, we need to understand the mapping properties of the linearisation of the scalar curvature operator. Moreover, when constructing approximate solutions on $\P(\cE)$, it is convenient to make changes to the $\dbar$-operator on the vector bundle. This uses the linearisation of the Hermite-Einstein operator, i.e.\ the contraction of the curvature, which is given by the Laplacian on endomorphisms of $\cE$. We therefore also describe the mapping properties of this operator in this section.

Let $\| \cdot \|_{L^2(g)}$ denote the $L^2$-norm with respect to a Riemannian metric $g$. We use that $\cE$ admits a filtration 
$$
0=\m{S}_0 \subset \m{S}_1 \subset \m{S}_2 \subset \ldots \subset \m{S}_d = \m{E}
$$
such that $\cE_0 = \oplus_{i=1}^d \m{Q}_i$ is a direct sum of \emph{simple} bundles $\m{Q}_i = \frac{\m{S}_i}{\m{S}_{i-1}}$. We also assume that $\mrm{Aut}(B,L)$ is discrete. In this section, we do not use the existence of any special metrics on either $\cE_0$ or $B$.

\subsection{The linearisation of the scalar curvature operator}
Let $L^2_q(g_{k,v})$ denote the Sobolev space of functions on $M$ whose weak derivatives up to order $q$ are in $L^2$, with respect to $g_{k,v}$. The main result of this section is a uniform bound for a right inverse to the Lichnerowicz operator.
\begin{proposition}
\label{prop:lichnerowicz}
There exists $C, k_0, \delta > 0$ such that for all $v$ such that $|v| \leq \delta$ and $k\geq k_0$ the operator 
$$L_{k,v} : L^2_{q+4}(g_{k,v}) \times \mfk \to L^2_q(g_{k,v})$$ 
given by 
$$
(\phi, \sigma) \mapsto \mathcal{D}^*_{g_{k,v}}\mathcal{D}_{g_{k,v}}(\phi ) - f_{\sigma} - \nu_{\sigma} (\phi)
$$
is surjective and has a right inverse $Q_{k,v}$ satisfying 
$$
\| Q_{k,v} \| \leq C k^3.
$$
\end{proposition}

We follow the strategy of \cite[Theorem 6.1]{fine04}. The primary goal is to give a Poincar\'e inequality for the map $\phi \mapsto \bar \partial_v (\nabla_{g_{k,v}} \phi)$, where $g_{k,v}$ is the Riemannian metric induced by $\omega_k$ and $J_v$. This gives a lower bound for the first non-zero eigenvalue of the (negative of) the Lichnerowicz operator, which in turn gives the required bound for $Q_{k,v}$.

To prove this Poincar\'e inequality, it is convenient to work with a product metric. The form $\omega$ defines a splitting $TX = TB \oplus V$. Using this splitting, we have a product metric $h_k = (k \omega_B) \oplus (\omega\big|_{V})$. As in Fine's work, this metric is uniformly equivalent to $g_{k,v}$, and we can make this independent of $v$ in any sufficiently small ball about $0$.
\begin{lemma}[{\cite[Lemma 6.2]{fine04}}]
\label{lem:uniformmetrics}
For any tensor bundle on $X$, there exists a $\delta> 0$ such that $g_{k,v}$ and $h_k$ are uniformly equivalent, uniformly for all $k$ sufficiently large and all $v$ such that $|v| < \delta$.
\end{lemma}

We begin with a Poincar\'e inequality for the exterior derivative on $X$. Note that on functions, the $L^2$-norm only depends on the symplectic form $\omega_k$, not the complex structure $J_v$.
\begin{lemma}[{\cite[Lemma 6.5]{fine04}}]
\label{lem:dPI}
There exists a $C>0$ such that for all $\phi$ of mean value zero with respect to $\omega_k$, 
$$
\| d\phi\|^2_{L^2(g_{k,v})} \geq Ck^{-1} \| \phi \|^2_{L^2(\omega_k)}
$$
for all $k$ sufficiently large.
\end{lemma}
This follows as in Fine's case, as he uses estimates with respect to the $h_k$, which also in our setting is uniformly equivalent to the metrics $g_{k,v}$. The proof is also similar to the proof of Lemma \ref{lem:dbarestimate}, which we give in full below.

The next step is to give a similar bound for $\dbar$ on global sections of $TX$, and we begin by describing this kernel, similarly to \cite[Lemma 6.3]{fine04}. Here we use the hypothesis that the group $\Aut(B,L)$ of automorphisms of $B$ that lift to $L$ is discrete. 
\begin{lemma}
\label{lem:automs}
Suppose $\Aut(B,L)$ is discrete. Then the holomorphy potentials on $X_v$ with respect to $\omega_{k,v}$ are precisely those induced by elements of $H^0(\End \m{E}_v)$ through the map $m$.
\end{lemma}

We can now prove the Poincar\'e estimate for $\db$. Note that this operator and even its kernel depend on $v$. We will therefore work orthogonally to the kernel $\mfh$ at $v=0$. Note that $\mfh$ may include holomorphic vector fields that do not admit holomorphy potentials, but these will not impact the main result of this section.
\begin{lemma}
\label{lem:dbarestimate}
There exists a $C>0$ such that for all $v$ sufficiently close to $0$ and all $k$ sufficiently large, 
$$
\| \dbar_v \nu \|^2_{L^2(g_{k,v})} \geq C k^{-2} \| \nu \|^2_{L^2(g_{k,v})}
$$
for all $\nu$ that are $g_{k,v}$-orthogonal to $\mfh \subset \Gamma(TX)$.
\end{lemma}
The proof follows closely that of \cite[Lemma 6.6]{fine04}. We will use the following comparison of the $h_k$ and $h_1$-norms several times. 
\begin{lemma}
\label{lem:scalingcomparison}
The $L^2$-norms with respect to the $h_k$ satisfy
\begin{align*}
k^{-n-1} \|\nu\|^2_{h_k} \leq \|\nu\|^2_{h_1} &\leq k^{-n} \|\nu \|^2_{h_k} \\
k^{-n} \|\alpha\|^2_{h_k} \leq \|\alpha\|^2_{h_1} &\leq k^{1-n} \|\alpha \|^2_{h_k} \\
k^{-n-1} \|\tau\|^2_{h_k} \leq \|\tau\|^2_{h_1} &\leq k^{1-n} \|\tau \|^2_{h_k} 
\end{align*}
on $TX$, $\Omega^1(X)$ and $\Omega^1(TX)$, respectively.
\end{lemma}
\begin{proof}
On $TX$, we have the following comparison of the pointwise norms
$$
k^{-1} |\nu|^2_{h_k} \leq |\nu|^2_{h_1} \leq |\nu |^2_{h_k}.
$$
On $\Omega^1(X)$, we have 
$$
|\beta|^2_{h_k} \leq |\beta|^2_{h_1} \leq k |\beta |^2_{h_k}.
$$
Combining the two gives the pointwise comparison 
$$
k^{-1} |\tau |^2_{h_k} \leq |\tau|^2_{h_1} \leq k |\tau |^2_{h_k}.
$$
of norms on $\Omega^1(TX)$. Since the volume form scales as $\vol(h_k) = k^n \vol(h_1),$ the result follows.
\end{proof}
We can now prove Lemma \ref{lem:dbarestimate}.
\begin{proof}[Proof of Lemma \ref{lem:dbarestimate}]
Let  $\Delta^h_v$ denote the Laplacian associated to the metric $h=h_1$ and operator $\dbar_v$, and let $\lambda$ be the first non-zero eigenvalue of the Laplacian $\Delta^h_0$. Note that the kernel of this Laplacian is the kernel of $\dbar_0$, which is $\mfh$. On the $h_1$-orthogonal complement to $\mfh$, $\lambda$ can be characterised as
\begin{align*}
\lambda = \inf_{\eta \in \mfh^{\perp}} \frac{\langle \Delta^h_0(\eta), \eta \rangle_h}{\| \eta \|^2_h},
\end{align*}
which is equivalent to the bound
$$
\| \dbar_0 (\eta) \|^2_{L^2(h_1)} \geq \lambda \| \eta \|^2_{L^2(h_1)}
$$
for all such $\eta$. Now, 
\begin{align*}
\langle \Delta^h_v(\eta), \eta \rangle_h =& \| \dbar_v (\eta) \|^2_{h_1} + \| \dbar^*_v(\eta) \|^2_{h_1},
\end{align*}
where the adjoint $\dbar^*$ is with respect to $h_1$. This varies continuously with $v$, and so, in particular, we see that we can ensure that
\begin{align*}
\inf_{\eta \in \mfh^{\perp}} \frac{\langle \Delta^h_v(\eta), \eta \rangle_h}{\| \eta \|^2_h}  \geq \frac{\lambda}{2}
\end{align*}
for all sufficiently small $v$, which in turn implies the bound 
\begin{align}
\label{eq:vPI}
\| \dbar_v (\eta) \|^2_{L^2(h_1)} \geq \frac{\lambda}{2} \| \eta \|^2_{L^2(h_1)}
\end{align}
on the $h_1$-orthogonal complement to $\mfh$ for all such $v$.

Next, let $\xi \in \mfh$ be such that $\nu - \xi$ is $h_1$-orthogonal to $\mfh$. If $\eta_1, \ldots, \eta_d$ is an $L^2$-orthonormal basis of $\mfh$ with respect to $h_1$, then the map $\nu \mapsto \xi$ is given by
\begin{align*}
\nu \mapsto \sum_j \left(\int_X \langle \nu, \eta_j \rangle_{h_1} \vol(h_1)\right) \eta_j.
\end{align*}
Since the $\eta_j$ are sections of the vertical tangent bundle, their pointwise $h_k$-norms and inner products are independent of $k$. Thus, the $\eta_j$ form an $L^2$-orthogonal basis with respect to the $h_k$-norm, and the norm is the scaling factor $k^n$ of the volume element. We therefore see that there is a $C_1>0$ such that
\begin{align*}
\| \xi \|^2_{L^2(h_k)} =& \sum_j \int_X \left(\int_X \langle \nu, \eta_j \rangle_{h_1} \vol(h_1)\right)^2 |\eta_j|_{h_k}^2 \vol(h_k) \\
=&  k^{-2n} \sum_j \int_X \left(\int_X \langle \nu, \eta_j \rangle_{h_k} \vol(h_k)\right)^2 |\eta_j|_{h_k}^2 \vol(h_k) \\
=& k^{-n} \sum_j \frac{\langle \nu , \eta_j \rangle_{L^2(h_k)}}{\| \eta_j \|_{L^2(h_k)}} \\
\leq& C_1 k^{-n} \| \nu \|^2_{L^2(h_k)},
\end{align*}
since the penultimate line is the $L^2(h_k)$-orthogonal projection of $\nu$ to $\mfk$.
Now,
\begin{align*}
\| \dbar_v \nu \|^2_{L^2(g_{k,v})} \geq& C_2 \| \dbar_v \nu \|^2_{L^2(h_k)},
\end{align*}
for some constant $C_2>0$, by Lemma \ref{lem:uniformmetrics}. Combining this with the above and using that $\dbar_0$ vanishes on $\mfk$, we see that by restricting to sufficiently small $v$, we can ensure that 
\begin{align}
\label{eq:nubound}
\| \dbar_v \nu \|^2_{L^2(g_{k,v})} \geq& C_2 \left( \| \dbar_v (\nu-\xi) \|^2_{L^2(h_k)} - \frac{\lambda}{4}k^{-n} \| \nu \|^2_{L^2(h_k)} \right).
\end{align}

We now seek to estimate $\| \dbar_v (\nu-\xi) \|^2_{L^2(h_k)}$ from below. By Lemma \ref{lem:scalingcomparison}, $$\| \dbar_v (\nu-\xi) \|^2_{L^2(h_k)} \geq k^{n-1} \| \dbar_v (\nu-\xi) \|^2_{L^2(h_1)}.$$
Using the bound \eqref{eq:vPI}, we therefore get that
\begin{align*}
\| \dbar_v (\nu-\xi) \|^2_{L^2(h_k)} \geq  \frac{\lambda}{2}k^{n-1}  \| \nu-\xi \|^2_{L^2(h_1)},
\end{align*}
and so, using Lemma \ref{lem:scalingcomparison} again to convert to a bound for the $h_k$, we get that
\begin{align*}
\| \dbar_v (\nu-\xi) \|^2_{L^2(h_k)} \geq  \frac{\lambda}{2} k^{-2} \| \nu-\xi \|^2_{L^2(h_k)}.
\end{align*}
Combining this with the bound \eqref{eq:nubound}, we get that 
\begin{align*}
\| \dbar_v \nu \|^2_{L^2(g_{k,v})} \geq&  \frac{C_2\lambda}{2}  \left( k^{-2} \| \nu-\xi \|^2_{L^2(h_k)} - \frac{1}{2}k^{-n} \| \nu \|^2_{L^2(h_k)} \right).
\end{align*}
Since by Lemma \ref{lem:uniformmetrics} $h_k$ and $g_{k,v}$ are uniformly equivalent, we therefore get that there is a $C>0$ such that 
\begin{align*}
\| \dbar_v \nu \|^2_{L^2(g_{k,v})} \geq&  2C \left( k^{-2} \| \nu-\xi \|^2_{L^2(g_{k,v})} - \frac{1}{2}k^{-n} \| \nu \|^2_{L^2(g_{k,v})} \right)
\end{align*}
for all sufficiently small $v$ and sufficiently large $k$.
But since $\nu$ is $g_{k,v}$-orthogonal to $\xi$, we have that 
\begin{align*}
\| \nu-\xi \|^2_{L^2(g_{k,v})} =& \| \nu \|^2_{L^2(g_{k,v})} + \| \xi \|^2_{L^2(g_{k,v})}\\
\geq& \| \nu \|^2_{L^2(g_{k,v})}.
\end{align*}
Thus
\begin{align*}
\| \dbar_v \nu \|^2_{L^2(g_{k,v})} 
\geq&  2C \left( k^{-2} \| \nu  \|^2_{L^2(g_{k,v})} - \frac{1}{2}k^{-n} \| \nu \|^2_{L^2(g_{k,v})} \right) \\
\geq& Ck^{-2}  \| \nu  \|^2_{L^2(g_{k,v})},
\end{align*}
since $n\geq 2$. This is what we wanted to show.
\end{proof}

We now combine the pieces as in \cite[Lemma 6.7]{fine04}. Recall that $\bar \mfk$ consists of the fibrewise holomorphy potentials induced by holomorphic sections of $\End \cE_0$ through $m$, which by Lemma \ref{lem:automs} are the global fibrewise holomorphy potentials with respect to $\omega_k$ for all $k$, with respect to the complex structure when $v=0$. 
\begin{lemma}
\label{lem:DPI}
There exists a $C>0$ such that for all $\phi$ that are $g_{k,v}$-orthogonal to $\bar \mfk \subset C^{\infty}(X)$, 
$$
\| \mathcal{D}_{k,v} \phi \|^2_{L^2(g_{k,v})} \geq C k^{-3} \| \phi \|^2_{L^2(g_{k,v})}
$$
for all $v$ that are sufficiently close to $0$ and sufficiently large $k$.
\end{lemma}
\begin{proof}
If $\phi$ is $g_{k,v}$-orthogonal to $\bar \mfk$, then $\phi$ has mean value $0$ with respect to $\omega_k$, and $\nabla_{g,k} \phi$ is $g_{k,v}$-orthogonal to $\mfh$ -- note that $\mfh$ may be larger than $\bar \mfk$, but any gradient vector field is automatically orthogonal to a holomorphic one that is also orthogonal to $\bar \mfk$. Hence Lemmas \ref{lem:dPI} and \ref{lem:dbarestimate} apply to give that for some constants $C', C>0$,
\begin{align*}
\| \mathcal{D}_{k,v} \phi \|^2_{L^2(g_{k,v})} \geq&C'k^{-2} \| \nabla_{g,k} \phi \|^2_{L^2(g_{k,v})}  \\
=& C'k^{-2} \| d \phi \|^2_{L^2(g_{k,v})}  \\
\geq& C k^{-3} \| \phi \|^2_{L^2(g_{k,v})},
\end{align*}
as required.
\end{proof}
\begin{remark}
At $v=0$, this is a lower bound for the first eigenvalue of the Lichnerowicz operator. However, in general this is not a lower bound for the first eigenvalue of the Lichnerowicz operator of $(\omega_k, J_v)$, but rather a bound for the first eigenvalue orthogonally to $\mfk$. When $v \neq 0$, the kernel of the Lichnerowicz operator gets smaller, and the new elements that are not in the kernel contribute to the first eigenvalue in a manner that depends on the ``discrepancy order" \cite{sektnantipler24} between $E$ and its subbundles coming from a Jordan--H\"older filtration, which was a crucial feature in \cite{sektnantipler24}. In our approach to semistable perturbation problems, we do not have to consider this.
\end{remark}

The bound for the right inverse to the Lichnerowicz operator follows from the above Poincar\'e inequality and the following Schauder estimate.
\begin{lemma}
\label{lem:schauder}
For each $q$, there exists a constant $C>0$ such that for all sufficiently small $v$ and sufficiently large $k$, 
$$
\| \phi \|_{L^2_{q+4}(g_{k,v})} \leq C \left( \| \phi \|_{L^2(g_{k,v})} + \| L \phi \|_{L^2_{q}(g_{k,v})}\right),
$$
where $L$ can be either the Lichnerowicz operator of $g_{k,v}$ or the actual linearisation $L_{k,v}$ of the scalar curvature operator at $g_{k,v}$.
\end{lemma}
The proof is similar to \cite[Lemma 5.9 and 6.8]{fine04}. 
Note that bounding one of the two operators is sufficient, because the scalar curvature is constant to leading order in $k$ and so is $O(k^{-1}|v|)$ in $k$ and $v$.

With this in place, we can now prove Proposition \ref{prop:lichnerowicz}, similarly to \cite[Theorem 6.9]{fine04}.
\begin{proof}[Proof of Proposition \ref{prop:lichnerowicz}]
Since $\ker \mathcal{D}_{g_{k,v}} \subseteq \ker \mathcal{D}_{g_{k,0}} = \bar \mfk$, the operator $L_{k,v}$ is surjective. We take the right inverse to be the one which gives the $L^2(g_{k,v})$-orthogonal representative to $\bar \mfk$ in the $\phi$-variable. Let now $(\psi, \sigma) = Q_{k,v} \phi$. Note that $\sigma$ is the $L^2(g_{k,v})$-orthogonal projection of $\phi$ to $\bar \mfk$, and so
$$
\| \sigma \|_{L^2(g_{k,v})} \leq \| \phi \|_{L^2(g_{k,v})}.
$$
Moreover, since $\sigma$ lives in the finite dimensional space $\bar \mfk$, we can for each $j$ mutually bound the $C^{j+4, \alpha}$-norm of $\sigma$ with the $L^2$-norm. 

From the Schauder estimate of Lemma \ref{lem:schauder}, we get that
\begin{align*}
\| \psi \|_{L^2_{q+4}(g_{k,v})} \leq C \left( \| \psi \|_{L^2(g_{k,v})} + \| \phi \|_{L^2_{q}(g_{k,v})}\right).
\end{align*}
By Lemma \ref{lem:DPI}, 
\begin{align*}
 \| \psi \|^2_{L^2(g_{k,v})} \leq& Ck^3 \| \mathcal{D}_{k,v} \psi \|^2_{L^2(g_{k,v})} \\
=&  Ck^3 \langle \mathcal{D}^*_{k,v} \mathcal{D}_{k,v} \psi, \psi \rangle_{L^2(g_{k,v})} \\
=& Ck^3 \langle \phi, \psi \rangle_{L^2(g_{k,v})} \\
\leq& Ck^3 \| \phi \|_{L^2(g_{k,v})} \| \psi \|_{L^2(g_{k,v})},
\end{align*} 
which gives that 
\begin{align*}
 \| \psi \|_{L^2(g_{k,v})} \leq& Ck^3 \| \phi \|_{L^2(g_{k,v})}.
\end{align*} 
So
\begin{align*}
\| Q_{k,v} \phi \|_{L^2(g_{k,v})} &\leq \| \psi \|_{L^2(g_{k,v})} + \| \sigma \|_{L^2(g_{k,v})}\\
&\leq C k^3 \| \phi \|_{L^2(g_{k,v})} + \| \phi \|_{L^2(g_{k,v})} \\
&\leq (C+1)k^3 \| \phi \|_{L^2(g_{k,v})}.
\end{align*}
Combined with the Schauder estimate above, this gives the result.
\end{proof}

\subsection{The Laplacian on the vector bundle}

We also require a similar bound for the Laplacian on the vector bundle, which is the linearised operator for the contraction of the curvature on $E$.
\begin{proposition}
\label{prop:bundleLaplacian}
There exists $C, \delta > 0$ such that for all $v$ such that $|v| \leq \delta$ the map
$$
L^2_{q} (\End E) \times \mf{k} \to L^2_{q}(\End E)
$$
given by
$$
(s, \sigma) \mapsto \Delta_{v} (s) - \sigma
$$
is surjective and has right inverse $R_{v}$ satisfying 
$$
\| R_{v} \| \leq C.
$$
\end{proposition}
\begin{proof}
The proof is very similar to the proof of Proposition \ref{prop:linearisation}, but easier since the metric $\omega_B$ on $B$ is fixed. Since $\mfk$ is the kernel of the Laplacian of $\Delta_0$, we can use a Poincar\'e inequality for the $\dbar_0$-Laplacian to obtain the uniform estimate
$$
\| \bar \partial_v s \|^2_{L^2(\omega_B)} \geq C \| s \|^2_{L^2(\omega_B)}
$$
for all $s$ that are orthogonal to $\mfk$, similarly to Lemma \ref{lem:dbarestimate}. Together with a Schauder estimate for $\Delta_v$ analogous to Lemma \ref{lem:schauder}, this gives the required bound.
\end{proof}

\section{Reduction to a finite dimensional moment map problem}
\label{sec:reduction}

We will now go back to consider a holomorphic vector bundle $\cE \to B$  of rank $r+1$ which is slope semistable with respect to some ample line bundle $L \to B$, and let $X=\mathbb{P}(\cE) \to B$ be the projectivisation of $\cE$. As in \S \ref{sec:projvb}, let $M$ be the smooth underlying manifold.

\subsection{The moment map for the initial metrics}
Let $\cV$ be the Kuranishi space about $\cE_0$ as in Section \ref{Subsec:diffgeomvectorbundles}, $\Phi : \cV \to \cJ_\pi$ be the Kuranishi map \eqref{eq:Kuranishimap_proj}, and $\m{V}^{\mrm{int}}$ be the complex subspace of $\m{V}$ parametrising integrable complex structures.

\begin{definition}[Weil--Petersson form]\label{Def:WeilPetersson}
Let $Y \to V$ be a proper holomorphic submersion of relative dimension $r$ and let $\beta$ be a fibrewise K\"ahler metric. Let $\rho$ be the relative Ricci form induced by $\beta$, i.e.\ the curvature of the Hermitian metric induced by $\beta$ on $-K_{Y/V}$. The Weil--Petersson form induced by $\beta$ is the 2-form on $V$
\[
\alpha_{WP} = - \int_{Y/V} \rho \wedge \beta^{r} + \frac{1}{r+1} \widehat{S}_x \int_{Y/V}\beta^{r+1},
\]
where $\widehat{S}_x$ denotes the average scalar curvature on the fibre $Y_x $, which is independent of $x$.
\end{definition}

We show that there is a closed two form $\alpha_k$, depending on $k$, on $\cV$, and a corresponding moment map for the $K$-action on $\cV$ such that having a zero of the moment map is related to solving the cscK problem on the corresponding fibre. More precisely, we apply the following.

\begin{theorem}[{\cite[Theorem 4.9]{DervanHallam23}}]
\label{thm:dervanhallam}
Suppose $Y \to V$ is a holomorphic submersion, with compact fibres of dimension $m$. Suppose that a group $K$ acts on $Y$ and $V$ such that the submersion is $K$-equivariant.  Suppose that $\beta$ is a relatively K\"ahler metric on $Y$ and that $\tau$ is a moment map with respect to $\beta$ of the $K$-action on $Y$. Then the map $\mu : V \to \mathfrak{k}^*$ defined by
$$
\langle \mu , \nu \rangle (v) = \int_{y \in Y_v} \langle \tau (y), \nu \rangle (\Scal(\beta_v) - \hat S) \beta^r
$$
is a moment map for the $K$-action on $V$, with respect to the closed two form $\alpha_{WP}$.
\end{theorem}

\begin{remark}
In the statement of Theorem \ref{thm:dervanhallam} we have the opposite sign compared to Dervan--Hallam's original statement. This is because we have chosen the definition \eqref{eq:infinitesimalvectorfield} for the infinitesimal vector field, while they chose to take the opposite sign. More concretely, this is reflected in the fact that the test configuration described in Remark \ref{Rmk:explicit_testconf}, which is the one to which we eventually apply Theorem \ref{thm:dervanhallam}, is induced by the element $-i\mrm{Id}_S \in \mf{k}$.
\end{remark}

We wish to apply the above to the action of $K$ on the universal family $\cU\to\cV$ defined in \S\ref{subsubsec:kuranishi_proj}, obtained from the Kuranishi map $\cV \to \cJ_{\pi}$. For each $k$, the metric $\omega_k = \omega + k \omega_B$ on $X$ gives a corresponding space $\cJ=\cJ_k$ of compatible almost complex structures which $\cJ_{\pi}$ embeds into. We can equally regard $\cU$ as being obtained from the corresponding map $\cV \to \cJ$.
While the almost complex structure on $\cJ$ is integrable, the universal family over it is not, in general. However, in the same work as above, Dervan--Hallam show \cite[Theorem 5.8]{DervanHallam23} that the same result of Theorem \ref{thm:dervanhallam} holds for the universal family over $\cJ$, provided one uses the Chern connection on the holomorphic tangent bundle to define the curvature quantities. The Weil--Petersson form one then obtains on $\cJ$ through Definition \ref{Def:WeilPetersson} is a closed (1,1)-form.

Now, the map $\cV \to \cJ$ is \emph{holomorphic}. Thus, the corresponding sequence of Weil--Petersson forms $\alpha_k$ on $\cV$ pulled back via this map are closed (1,1)-forms, since they are pulled back via a holomorphic map.
%
More explicitly, we can also see these forms as being obtained in the following way. From the sequence of K\"ahler metrics on $M$, we have relatively almost K\"ahler forms on $\cU$, denoted $\omega_{\m{U},k}$. These are obtained as the pull-back of $\omega_k$ from $M$ via the projection $\mrm{pr}_2 : \m{V}\times M \to M$, where we have used that $\m{U}$ is diffeomorphic to the product $\m{V}\times M$.
We can then write 
\[
\omega_{\m{U},k} = \widehat{\omega} + k\omega_B
\]
where $\widehat{\omega}$ is the pullback to $\m{U}$ of the relatively K\"ahler form $\omega$ on $M \to B$. The form $\alpha_k$ is then given by
\[
\alpha_k = -\int_{\m{U}/\m{V}} \rho_{\m{U},k} \wedge \omega_{\m{U},k}^{n+r} + \frac{1}{n+r+1}\int_{\m{U}/\m{V}}\widehat{S}_v \omega_{\m{U},k}^{n+r+1}
\]
on $\m{V}$.
%
%
%
%
%
%
%
We next show that these are K\"ahler in the adiabatic regime.
\begin{lemma}
\label{lem:WPKahler}
After possibly shrinking $\m{V}$, the Weil--Petersson form $\alpha_k$ is K\"ahler for all sufficiently large $k$.
\end{lemma}
\begin{proof}
As $\alpha_k$ is a closed (1,1)-form on $\m{V}$, we need to show positivity for all sufficiently large $k$. All we need for our application of the result is that the $\alpha_k$ are positive on $\m{V}^{\mrm{int}}$ for all sufficiently large $k$. For this reason, for simplicity we show the claim holds over the integrable locus $\m{V}^{\mrm{int}}$. All the structures involved in the definition of the $\alpha_k$ for non-integrable complex structures are explained in \cite[Section 5]{DervanHallam23}.

The relatively symplectic form $\omega_{\m{U},k}$ is constant in the horizontal direction with respect to the projection $\mrm{pr}_1:\m{U} \to \m{V}^{\mrm{int}}$. So the second term in the expression of $\alpha_k$ vanishes.
Thus we concentrate on the first term.
We expand in powers of $k$ the expression of $\rho_{\m{U},k}$ as follows:
\[
\begin{aligned}
\rho_{\m{U},k} &= i\del\delbar\log \left(k^n\widehat{\omega}^r\wedge\omega_B^n + \Ok{n-1}\right)\\
&= i\del\delbar\log \left(\widehat{\omega}^r\wedge\omega_B^n\right) + \Ok{-1}\\
&= i\del\delbar \log \det(\widehat{\omega}) + i \del \delbar \log\det(\omega_B) + \Ok{-1}.
\end{aligned}
\]
The $k^{-1}$-term is exact on $\m{U}$ because $\omega_{\m{U},k}^{n+r}$ and $\widehat{\omega}^r\wedge\omega_B^n$ induce the same class on $-K_{\m{U}/\m{V}^{\mrm{int}}}$.
Therefore we can write
\[
\alpha_k = -\int_{\m{U}/\m{V}^{\mrm{int}}} \rho_{\m{U},0} \wedge \omega_{\m{U},k}^{n+r}.
\]
Using the expansion of Equation \eqref{eq:adiabaticvolform} for the volume form, the leading order term in the expansion $\sum_{j} \hat \alpha_j k^{n-j}$ of $\alpha_k$ is
\[
\hat \alpha_0 = -\int_{\m{U}/\m{V}^{\mrm{int}}} i\del\delbar\log\det(\widehat{\omega}) \wedge \widehat{\omega}^r\wedge\omega_B^n,
\]
which vanishes because, as before, the form $\widehat{\omega}$ is constant in the horizontal-with-respect-to-$\pi_{\m{U}}$ direction.
This is consistent with the fact that the leading order term of the Weil--Petersson metric is the Weil--Petersson metric of the fibres of $\m{U} \to B \times \m{V}^{\mrm{int}}$, which are projective spaces, hence rigid.

We then turn our attention to the subleading order term
\[
\hat \alpha_1 = -\int_{\m{U}/\m{V}^{\mrm{int}}}i\del\delbar\log\det(\widehat{\omega}) \wedge \widehat{\omega}^{r+1}\wedge\omega_B^{n-1}-\int_{\m{U}/\m{V}^{\mrm{int}}} \mrm{Scal}(\omega_B) \widehat{\omega}^{r+1}\wedge\omega_B^n.
\]
Again, the second addendum vanishes because the form $\widehat{\omega}$ is constant in the horizontal-with-respect-to-$\pi_{\m{U}}$ direction.
We have proven that
\[
\alpha_k = -k^{n-1}\int_{\m{U}/\m{V}^{\mrm{int}}}i\del\delbar\log\det(\widehat{\omega}) \wedge \widehat{\omega}^{r+1}\wedge\omega_B^{n-1} + \Ok{n-2}
\]
It remains to prove that the $k^{-1}$-term is equal - up to a multiplicative constant - to the Weil--Petersson metric on the Kuranishi space $\m{V}^{\mrm{int}}$ of vector bundles centred about the polystable vector bundle $\m{E}_0\to B$. Let $\m{F} \to \m{V}^{\mrm{int}} \times B$ the universal vector bundle. We consider the Weil--Petersson metric
\[
\alpha_{WP, \m{F}} = \int_{\m{V}^{\mrm{int}}\times B/\m{V}^{\mrm{int}}} \mrm{tr}(F \wedge F) \wedge \omega_B^{n-1},
\]
where $F$ is the curvature of the universal connection on $\m{F} \to \m{V}^{\mrm{int}} \times B$.
Going back to the expression of $\alpha_k$, we note that
\[
i\del\delbar\log\det(\widehat{\omega}) \wedge \widehat{\omega}^{r+1}\wedge\omega_B^{n-1} = \left(i\del\delbar\log\det(\widehat{\omega})\right)_{\m{H}} \wedge \widehat{\omega}_{\m{H}}\wedge\widehat\omega^{r}\wedge\omega_B^{n-1}
\]
where the foot index $_\m{H}$ indicates the horizontal restriction \emph{with respect to $B$}.
It follows from \cite[Lemma 3.2]{dervansektnan21a} that
\[
\widehat{\omega}_{\m{H}} = m^*F + \pi^*\beta,
\]
for some 2-form $\beta$ on $B$.
Since the fibres are projective spaces,
\[
\left(i\del\delbar\log\det(\widehat{\omega})\right)_{\m{H}} = m^* F,
\]
see \cite[Proposition 3.17]{dervansektnan21a}, \cite[Lemma 5.4]{Bronnle_PhDthesis}.
Using the expression of $m$ \eqref{Eq:CinftyE_correspondence} we see that, as a $2$-form on $\m{V}^{\mrm{int}}\times B$,
\[
m^*F = \mrm{Tr}(F)
\]
so
\[
\begin{aligned}
\hat \alpha_1 &=\int_{\m{U}/\m{V}^{\mrm{int}}}i\del\delbar\log\det(\widehat{\omega})_{\m{H}} \wedge \widehat{\omega}_{\m{H}}\wedge \widehat{\omega}^{r}\wedge\omega_B^{n-1} \\
&=\int_{\m{V}^{\mrm{int}}\times B/\m{V}^{\mrm{int}}} \mrm{Vol}(\bb{P}(\m{E}_b)) \mrm{Tr}(F)\wedge \mrm{Tr}(F)\wedge \omega_B^{n-1} \\
&=\mrm{Vol}(\bb{P}(\m{E}_b)) \int_{\m{V}\times B/\m{V}^{\mrm{int}}} \mrm{Tr}(F\wedge F)\wedge \omega_B^{n-1} + \mrm{Vol}(\bb{P}(\m{E}_b)) \int_{\m{V}^{\mrm{int}}\times B/\m{V}^{\mrm{int}}} \eta\wedge \omega_B^{n-1},
\end{aligned}
\]
where $\eta$ is a representative of the second Chern class $c_2(\m{F})$.
The first term is the Weil--Petersson metric induced by the universal vector bundle $\m{F}\to B \times \m{V}^{\mrm{int}}$.
As for the second term, since the second Chern class is pulled back from $B$, we can write
\[
\eta = \eta_0 + \dd\zeta
\]
where $\eta_0$ is a representative of $c_2(\m{E}_0)$ and $\zeta$ is a three-form on $B$ that depends on $\m{V}^{\mrm{int}}$.
In particular
\[
\int_{\m{V}^{\mrm{int}}\times B/\m{V}^{\mrm{int}}} \eta_0\wedge \omega_B^{n-1} = 0.
\]
Since the Weil--Petersson metric $\alpha_{WP, \m{F}}$ is positive on $\m{V}^{\mrm{int}}$, and $\zeta\vert_{v=0}=0$, up to shrinking $\m{V}^{\mrm{int}}$ we have that $\hat \alpha_1$ is positive.
So $\alpha_k$ is a positive form on $\m{V}^{\mrm{int}}$ and hence positive in a neighbourhood of $\m{V}^{\mrm{int}}$. The result follows.
\end{proof}

To understand what the moment map on $\cV$ with respect to $\alpha_k$ is, we need to understand the moment map of the $K$-action on $\cU$ with respect to $\omega_k$. The starting point for this is that we can view $\mathfrak{k} \subset C^{\infty} (M)$, as explained in \S \ref{Subsec:diffgeomvectorbundles}. 
Thus 
$$
\langle \tau(u), \sigma \rangle = h_{\sigma} (\mrm{pr}_2(u))
$$
is a moment map with respect to $\omega_k$ for the $K$-action on $\cU$ for \emph{all} $k$. Here $\mrm{pr}_2 : \cU \to M$ is the projection to $M$, using that $\cU = \cV \times M$ as a smooth manifold.

We are now ready to understand the moment map condition on $\cV$ induced from $\omega_k$ on $\cU$.
\begin{lemma}
\label{lem:momentmapzero}
A point $v \in \cV$ is a zero of the moment map $\mu_k$ if and only if the projection of the scalar curvature $S(\omega_k)$ of $\omega_k$ is orthogonal to the $v$-dependent embedding $\mathfrak{k} \subset C^{\infty}(M)$.
\end{lemma}
\begin{proof}
By Theorem \ref{thm:dervanhallam}, the moment map $\mu_k$ with respect to the Weil--Petersson form $\alpha_k$ on $\cV^{\mrm{int}}$ is given by
$$
\langle \mu_k(v), \sigma \rangle  = \int_{u \in \cU_v} \langle \tau(u), \sigma \rangle (\mrm{Scal}(\omega_k)\big|_{\cU_v} - \widehat S) \omega_{k,v}^{n+r},
$$
and it can be extended to a unique moment map $\mu_k$ on $\m{V}$.
By the above,
$$
\langle \tau(u), \sigma \rangle = h_{\sigma} (\pi_M (u)).
$$
Since $\cU_v \cong M$ via the projection $\pi_M$ and the holomorphic structure is that of $\P(\cE_v)$, we therefore get that
$$
\langle \mu_k(v), \sigma \rangle = \int_{\P(\cE_v)} h_{\sigma}(m) (\mrm{Scal}(\omega_k) - \hat S) \omega_k^{n+r}.
$$
Thus $v$ is a zero of the moment map if and only if $\mrm{Scal}(\omega_k, J_v) - \widehat S$ is orthogonal to the space of functions spanned by the $h_{\sigma}$. This is precisely the realisation of $\mathfrak{k}$ as a subspace of $C^{\infty}(M)$ which depends on $v$.
\end{proof}

\subsection{Approximate solutions}
Our goal now is to change the relatively K\"ahler structure on $\cU$ so that a zero of the corresponding moment map on $\cU$ genuinely corresponds to a cscK metric. In this section, as a step towards this, we compute the expansion of the scalar curvature of $(\omega + k\omega_B, J_v)$ in inverse powers of $k$ and then alter the K\"ahler structure to construct approximate solutions to solving that the scalar curvature lies in $\mfk$.

We begin by making the following assumptions on the geometry of the base:
\begin{enumerate}
\item the scalar curvature of $\omega_B$ is constant:
\[
\mrm{Scal}(\omega_B) = \widehat{S}_B;
\]
\item the automorphism group $\mrm{Aut}(B, L)$ is discrete.
\end{enumerate}

Using the isomorphism $m$ described in Lemma \ref{Lemma:CinftyE_correspondence}, the expansion \eqref{Eq:expansion_curvature_vb} of the curvature of the vector bundle, and the computations in \cite[Lemma 18]{bronnle15}, we obtain the following expansion of the scalar curvature of $(\omega_k := \omega + k\omega_B, J_v)$ to sub-leading order in $k$:
\begin{equation}\label{Eq:expansion_scalar_curvature}
\mrm{Scal}(\omega_k, J_v) = \widehat{S}_b + k^{-1} \left(\widehat{S}_B + m^*(\Lambda_{\omega_B}F_v) \right) + \Ok{-2}.
\end{equation}
In particular, the leading order term is constant since the fibres of $\m{E}_0 \to B$ all have constant scalar curvature given by the fibrewise Fubini-Study metric. Moreover, at $v=0$, the term $m^*(\Lambda_{\omega_B}F_v)$ vanishes, since we have a Hermite--Einstein metric on the central fibre.

The following lemma describes the $k$-expansion of the linearisation of the scalar curvature.

\begin{lemma}\label{Lemma:approximate_linearisation}
Let $\m{L}_{k,v}$ be the linearisation of the scalar curvature at a K\"ahler potential $\varphi$. Then
\[
\m{L}_{k,v}(\varphi) = \m{L}_0(\varphi) + k^{-1} \m{L}_1(\varphi) + k^{-2}\m{L}_2 + \Ok{-3},
\]
where:
\begin{enumerate}
\item $\m{L}_0(\varphi) =- \m{D}_{\m{V}}^*\m{D}_{\m{V}}(\varphi)$;
\item if $\varphi_B \in C^\infty(B)$, then $\m{L}_0(\varphi_B) = \m{L}_1(\varphi_B)=0$ and
\[
\int_{X/B}\m{L}_2(\varphi_B) \omega^{r}= -\m{D}_B^*\m{D}_B(\varphi_B);
\]
\end{enumerate}
\end{lemma}
\begin{proof}
Both results follow as in Fine \cite[\S3]{fine04}. Since the leading order term of the expansion \eqref{Eq:expansion_scalar_curvature} is the scalar curvature of the fibres, which is constant, the leading order term of its linearisation \eqref{eq:linearisation_scalar_curvature} is the Lichnerowicz operator of the fibres.
The same observation holds for the second claim. In fact, adding a potential $\phi_B$ to $\omega_{k,1}$ amounts to adding the potential $k^{-1}\phi_B$ to the base metric $\omega_B$.
Since the scalar curvature of the base affects the order $k^{-1}$-term and not the order zero term, the combined effect on the linearisation is of order $k^{-2}$.
\end{proof}

We are now ready to prove that the K\"ahler metric can be approximately modified by adding K\"ahler potentials so that the scalar curvature lies in the Lie algebra $\mf{k}$ up to any given order in $k$. On the non-integrable fibres below, we use $\frac{1}{2}(\dd\dd^c(\phi))^{1,1}$ to change the form. The form $\omega_k + \frac{1}{2}(\dd\dd^c(\phi))^{1,1}$ is then compatible with the almost complex structure, but not necessarily closed. On non-integrable deformations, by $\mrm{Scal}$ we mean the Chern scalar curvature of the resulting Hermitian metric.
\begin{proposition}
\label{prop:approxsolns}
For every $v \in \m{V}$ and for every $j$, there exist sections
\[
\sigma_1, \dots, \sigma_j \in \mf{k} \subset \Gamma(\End(E)), \quad s_1, \dots, s_j \in \mf{k}^\perp \subset \Gamma(\End(E)),
\]
functions
\[
\quad l_1, \dots, l_j \in C^\infty_R(X), \quad b_1, \dots, b_j \in C^\infty(B)
\]
and constants $c_1, \dots, c_j$ such that if
\[
\phi_{k,j} = \sum_{i=1}^j k^{-i}l_i + k^{-i+2}b_i,
\]
and 
\[
\dbar_{k,v,j} = \exp({k^{-j+1}s_r})\cdot \ldots \cdot \exp({k^{-1}s_2}) \cdot \exp{s_1}\cdot \dbar_v,
\]
then we have
\[
\mrm{Scal}(\omega_k + i\del\delbar (\phi_{k,j}), J_{k,v,j}) = \widehat{S}_b + \sum_{i=1}^j k^{-i} \left(h^{\phi_{k,j}}_i + c_i\right)  + \Ok{-j-1},
\]
where $J_{k,v,j}$ is the almost complex structure on $M$ induced by $\dbar_{k,v,j}$, for each $i$ the function $h_i$ is the Hamiltonian function corresponding to $\sigma_i$ and $h^{\phi_{k,j}}_i$ is given by \eqref{eq:changeinholpot}.
\end{proposition}
\begin{proof}
The proof is by induction on $j$. We begin with the base step $j=1$, i.e.\ with $k^{-1}$-term, and then we describe the induction step $j=2$.
Let $s \in \Gamma(\End E)$. Changing $\dbar_v$ to $\exp(s)\cdot \dbar_v$ changes the bundle curvature to
\[
\Lambda_{\omega_B}F_{\mrm{exp}(s)\cdot \dbar_{v}}.
\]
Note first that we can always solve the equation
\begin{equation}\label{Eq:bundlecurvature_in_k}
\Lambda_{\omega_B}F_{\mrm{exp}(s)\cdot \dbar_{v}} = \sigma \in \mf{k}
\end{equation}
for $s$ and $\sigma$, for all sufficiently small $v$. The linearisation of the condition \eqref{Eq:bundlecurvature_in_k} is given by the map
\[
(s,\sigma) \mapsto \Delta_v(s) -\sigma.
\]
Using Proposition \ref{prop:bundleLaplacian}, the implicit function theorem implies that there exists in fact a solution $(s_1, \sigma_1)$ to equation \eqref{Eq:bundlecurvature_in_k}. Here we also use that when $v=0$, we are already at a solution for this equation, since $h$ is a Hermite--Einstein metric for the complex structure at the central fibre. 

In particular, $\sigma_1$ corresponds via the map $m$ \eqref{Eq:CinftyE_correspondence} to a fibrewise holomorphy potential $\widetilde{h}_1 \in \mf{k}$ for the metric $(\omega_k, J_v)$.
Therefore, letting $J_{k,v,1}$ be the complex structure induced by $\dbar_{k,v,1} = \exp(s_1)\cdot \dbar_v$ (which is actually independent of $k$), we have that
\begin{align*}
\mrm{Scal}(\omega_k, J_{\exp(s_1)\cdot v}) =&  \widehat{S}_b + k^{-1} \left(\widehat{S}_B + m^*(\Lambda_{\omega_B}F_{\exp(s_1)\cdot \dbar_v}) \right) + \Ok{-2} \\
=& \widehat{S}_b + k^{-1} \left(\widehat{S}_B + \widetilde{h}_1 \right) + \Ok{-2}.
\end{align*}
By Lemma \ref{lem:changeinholpot}, the function
\begin{align*}
h^{\phi_{k,1}}_1 =& \widetilde{h}_1 + \langle \nabla \widetilde{h}_1 , \nabla {\phi_{k,1}} \rangle \\
=& \widetilde{h}_1 + \langle \nabla \widetilde{h}_1, \nabla k^{-1}l_1\rangle
\end{align*}
is a Hamiltonian with respect to $\omega_k+ i\del\delbar {\phi_{k,1}}$, where $\phi_{k,1} = k^{-1}l_1$. Let
\[
\omega_{k,1} =  \omega_k + i\del\delbar \phi_{k,1}.
\]
Since $h^{\phi_{k,1}}_1 - \widetilde{h}_1$ is $O(k^{-1})$,
we can write
\[
\mrm{Scal}(\omega_{k,1}, J_{k,v,1}) = \widehat{S}_b + k^{-1} \left(\widehat{S}_B + h^{\phi_{k,1}}_1 \right) + \Ok{-2},
\]
where $h^{\phi_{k,1}}_1$ is a Hamiltonian 
for a vector field in $\mfk$ with respect to $\omega_{k,1}$.

We next describe the induction step in the case $j=2$. 
We write explicitly
\[
\mrm{Scal}(\omega_{k,1}, J_{k,v,1}) = \widehat{S}_b + k^{-1} \left(\widehat{S}_B + h^{\phi_{k,1}}_1 \right) + k^{-2}\left(\psi_{2, B} + \psi_{2, E} + \psi_{2, R}\right)+ \Ok{-3}.
\]
We first perturb the K\"ahler metric so that the base term can be made constant. Let $b \in C^\infty(B)$.
From Lemma \ref{Lemma:approximate_linearisation}, the change of $\omega_{k,1}$ by the K\"ahler potential $b$ affects the scalar curvature as
\[
\mrm{Scal}(\omega_{k,1}+ i \del\delbar b, J_{k,v,1}) = \widehat{S}_b + k^{-1} \left(\widehat{S}_B + h^{\phi_{k,1}}_1 \right) + k^{-2}\left(\psi_{2, B} -\m{D}_B^*\m{D}_B(b) + \psi_{2, E}' + \psi_{2, R}'\right)+ \Ok{-3}.
\]
Since the automorphism group of $(B,L)$ is discrete, the equation
\[
\psi_{2,B} - \m{D}_B^*\m{D}_B(b) = constant
\]
admits a solution, which we denote by $b_2$. This makes the $C^\infty(B)$-term constant to order $k^{-2}$.

The corrections to the $C^\infty_W$-term and to the $C^\infty_R$-term now work as in the case $r=1$: we first produce a section $s_2 \in \mf{k}$ such that if we let $\dbar_{k,v,2} = \exp(s_2)\cdot\dbar_{k,v,1}$ and let $J_{k,v,2}$ be the corresponding almost complex structure on $M$, then
\[
\begin{split}
\mrm{Scal}(\omega_{k,1} + i\del\delbar b_2, J_{k,v,2}) = \widehat{S}_b + k^{-1} \left(\widehat{S}_B + h^{\phi_{k,1}}_1 \right) 
+ k^{-2}\left(c_{2} + \widetilde{h}_2 + \psi_{2,R}''\right)+ \Ok{-3},
\end{split}
\]
where $\widetilde{h}_2$ is a holomorphy potential for $(\omega_k, J_v)$.
Next, we eliminate the $C^\infty_R$-term, 
and we obtain an $l_2$ such that
\[
\begin{aligned}
\mrm{Scal}\left(\omega_{k,1} + i\del\delbar \left( b_2 + k^{-2}l_2 \right), J_{\exp(k^{-1}s_2)\cdot v_{k,1}}\right) = \widehat{S}_b &+ k^{-1} \left(\widehat{S}_B + h^{\phi_{k,1}}_1 \right) +\\
&+ k^{-2}\left(c_{2} + \widetilde{h}_2\right)+ \Ok{-3},
\end{aligned}
\]
where the change of the K\"ahler form by the potential $k^{-2}l_2$ does not affect the $C^\infty(B)$ and the $C^\infty_W$-terms at order $k^{-2}$.

Finally, we replace $h^{\phi_{k,1}}_1$ and $\widetilde{h}_2$ with the holomorphy potentials with respect to $\omega_{k,1} + i\del\delbar \left( b_2 + k^{-2}l_2\right)$, i.e. with $h_2^{\phi_{k,2}}$ and $h_2^{\phi_{k,2}}$, where $\phi_{k,2} = \phi_{k,1} + b_2 + k^{-2}l_2$. To do this, we first observe that, since $\sigma_j$ are vertical vector fields, $\sigma_j(b_2)= 0$, so the K\"ahler potential $b_2$ does not change the expression of the holomorphy potential. Therefore, using that
\begin{align*}
k^{-1} h_1^{\phi_{k,1}} - k^{-1} h_1^{\phi_{k,2}} = O(k^{-3})
\end{align*}
and
\begin{align*}
k^{-2} h_2^{\phi_{k,2}} - k^{-2}\widetilde{h}_2 = O(k^{-3}),
\end{align*}
we can replace $h_1^{\phi_{k,1}}$ by $h_1^{\phi_{k,2}}$ and $\widetilde{h}_2$ by $h_2^{\phi_{k,2}}$ in the above expansions. Setting
\[
\omega_{k,2} = \omega_{k,1} + i\del\delbar (b_2 + k^{-2}l_2),
\]
we then have that 
\[
\begin{split}
\mrm{Scal}(\omega_{k,2}, J_{ v_{k,2}}) = \widehat{S}_b + k^{-1} \left(\widehat{S}_B + h^{\phi_{k,2}}_1 \right) 
+ k^{-2}\left(c_{2} +h^{\phi_{k,2}}_2\right)+ \Ok{-3},
\end{split}
\]
which gives the result for $j=2$. Proceeding in the same way by induction on $j$ concludes the proof.
\end{proof}
\begin{remark}\label{rem:smoothlyvarying}
The above approximate solutions depend smoothly on $v \in \cV$, as the operators to which we have applied the implicit function theorem depend smoothly on $v$.
Note also that the arguments go through for $v$ outside the integrable locus in $\cV$. We are then considering the Chern scalar curvature rather than the Riemannian scalar curvature. All we need for the argument is that the linearised operator at $\bar \partial_v$ is an $O(|v|)$-perturbation of that of $\bar \partial_0$, which holds also in the non-integrable case, as Lemma \ref{lem:dbarestimate},  which is the key for Proposition \ref{prop:bundleLaplacian}, does not use integrability. For the same reason, the arguments in the section below also hold over the non-integrable locus, yielding the results on the whole of $\mathcal{V}$.  Note that in general the resulting form will not be closed, so the Hermitian metric on the non-integrable fibres is not almost K\"ahler, in contrast to our original metric before the perturbation.
\end{remark}

The above results are pointwise, but in fact can be shown to hold in any desired Sobolev space.
\begin{proposition}
\label{prop:approxsolnsSobolev}
Let $(\omega_k + i\del\delbar (\phi_{k,v,j}), J_{k,v,j})$ be the K\"ahler structure constructed  in Proposition \ref{prop:approxsolns}, and $h_i$ and $c_i$ the corresponding hamiltonians and constants. For each $j$ and for each $q$, there exists a $C>0$ such that for all sufficiently small $v$ and sufficiently large $k$,
\[
\| \mrm{Scal}(\omega_k + i\del\delbar (\phi_{k,v,j}), J_{k,v,j}) - \widehat{S}_b - \sum_{i=1}^j k^{-i} \left(h^{\phi_{k,j}}_i + c_i\right)  \|_{L^2_{q}(g_{k,v})} \leq  Ck^{\frac{n}{2}-j-1}.
\]
\end{proposition}
\begin{proof}
This follows as in \cite[Lemma 5.7]{fine04}. The norm of a fixed function is $O(1)$ in the $C^k(g_{k,v,j})$-norms \cite[Lemma 5.6]{fine04}, where $g_{k,v,j}$ is the Riemannian metric coming from the K\"ahler structure $(\omega_k + i\del\delbar (\phi_{k,v,j}), J_{k,v,j})$. From this, it follows that
$$
\| \mrm{Scal}(\omega_k + i\del\delbar (\phi_{k,v,j}), J_{k,v,j}) - \widehat{S}_b - \sum_{i=1}^j k^{-i} \left(h^{\phi_{k,j}}_i + c_i\right) \|_{C^q(g_{k,v,j})} \leq C k^{-j-1}.
$$ 
Incorporating the volume form, we have a general bound
$$
\| f \|_{L^2_{q}(g_{k,v,j})} \leq C \cdot \Vol(X, g_{k,v,j})^{\frac{1}{2}} \cdot \| f \|_{C^{q}(g_{k,v,j})},
$$
the square root coming from the fact that we are working with $L^2$-norms. Since the volume is bounded above by a constant multiple of $k^n$, the result follows.
\end{proof}

\subsection{Perturbing the K\"ahler metric}

Ultimately, we want to solve the cscK equation on $\P(\cE)$, not just that the scalar curvature is orthogonal to $\mathfrak{k}$. We now perturb the K\"ahler metrics $\omega_k$ to metrics whose scalar curvature is in $\mathfrak{k} \subset C^{\infty}(M)$. This is done so that having a zero of the corresponding moment map at $v \in \cV$ is equivalent to having a genuine cscK metric on $\mathbb{P}(\cE_v)$.

Recall from \S\ref{sec:mfk} that $\overline \mfk^M_{\psi} = \{ \widetilde h^{\psi}_{\sigma}  : \sigma \in \mfk \}$ and let $\Pi_{k,\psi}^{\perp}$ denote the projection to the $L^2$-orthogonal complement of $\overline \mfk^M_{\psi}$ in $C^{\infty}(M)$ with respect to $\omega_k + \ddb \psi$. Our goal is to show the following.
\begin{proposition}
\label{prop:findimreduction}
Let $J_{k,v} = J_{k,v,j}$ for a suitably chosen $j$. For all sufficiently large $k$ and sufficiently small $v$, there exists a $\phi_{k,v}$ such that
$$
\Pi^{\perp}_{k,\phi_{k,v}} (\Scal(\omega_k + \ddb \phi_{k,v}), J_{k,v}) - \hat S_k) = 0.
$$
\end{proposition}

We will prove this using the implicit function theorem. This uses the following.
\begin{proposition}
\label{prop:linearisation}
Let $(\omega_{k,v,j}= \omega_k + \ddb \phi_{k,v,j}, J_{k,v,j})$ be the metric structure on $M$ constructed in Proposition \ref{prop:approxsolns} and let $\sigma_{k,v,j} \in \mfk$ be such that $m^* \sigma_{k,v,j} = \sum_{i=1}^j k^{-i} h_i$. For each $j \gg 0$, there exists $C, k_0, \delta > 0$ such that for all $v$ such that $|v| \leq \delta$ and $k\geq k_0$ the linearisation $L_{k,v,j}$ of the map $$\Phi_{k,v,j} : L^2_{q+4}(g_{k,v}) \times \mfk \to L^2_{q}(g_{k,v})$$ given by
$$
(\phi, \sigma) \mapsto \Scal(\omega_{k,v,j} + \ddb(\phi), J_{k,v,j}) - \hat S_k - h^{\phi_{k,v,j}}_{\sigma} - \nu_{\sigma} (\phi)
$$
at $(0, \sigma_{k,v,j})$ is surjective and has right inverse $Q_{k,v,j}$ satisfying 
$$
\| Q_{k,v,j} \| \leq C k^3.
$$
\end{proposition}
\begin{proof}
Proposition \ref{prop:lichnerowicz} applies to the initial K\"ahler structure $(\omega_k, J_v)$. However, from the construction in Proposition \ref{prop:approxsolns} of $(\omega_{k,v,j}, J_{k,v,j})$, we see that for each $j$, the corresponding Lichnerowicz operators satisfy a similar bound as in Proposition \ref{prop:lichnerowicz}. Thus the required bound holds for the Lichnerowicz operator of $(\omega_{k,v,j}, J_{k,v,j})$. The actual linearised operator in the $\phi$ variable differs from the negative of the Lichnerowicz operator by the term
$$
\langle \nabla_{g_{k,v,j}} \Scal(\omega_{k,v,j}, J_{k,v,j}) , \nabla_{g_{k,v,j}} \phi \rangle_{g_{k,v,j}}.
$$
Now, 
\begin{align*} 
\Scal(\omega_{k,v,j}, J_{k,v,j}) =& \hat S_k + \sum_{i=1}^j k^{-i} h^{\phi_{k,v,j}}_i + O(k^{-j-1}) \\
=&  \hat S_k + h_{\sigma_{k,v,j}} + \nu_{\sigma_{k,v,j}}(\phi_{k,v,j}) + O(k^{-j-1}),
\end{align*}
and so
\begin{align*} 
\nabla_{g_{k,v,j}} \Scal(\omega_{k,v,j}, J_{k,v,j}) =& \nu_{\sigma_{k,v,j}} + O(k^{-j-1}).
\end{align*}
So when we linearise at $\sigma = \sigma_{k,v,j}$, we see that $L_{k,v,j}$ differs from the negative of the Lichnerowicz operator by a term which is $O(k^{-j-1})$. Thus, if $j$ is chosen sufficiently large, the linearised operator admits a similar bound, which is what we wanted to show.
\end{proof}
We will also need the following application of the mean value theorem.
\begin{lemma}
\label{lem:mvt}
Let $\Psi_{k,v,j} = \Phi_{k,v,j} - L_{k,v,j}$. There exists $c,C > 0$ such that for all sufficiently small $v$ and all $k \gg 0$, if $x,y \in L^2_{q+4}(g_{k,v}) \times \mfk$ satisfy $\| x \|_{L^2_{q+4}} < c$ and $\| y \|_{L^2_{q+4}} < c$, then 
\begin{align*}
\| \Psi_{k,v,j} (x) - \Psi_{k,v,j} (y) \|_{L^2_{q}} \leq C \left( \| x \|_{L^2_{q+4}} + \| y \|_{L^2_{q+4}} \right) \| x - y \|_{L^2_{q+4}}.
\end{align*}
\end{lemma}

The above allows us to use the following quantitative implicit function theorem to prove Proposition \ref{prop:findimreduction}.
\begin{theorem}[{\cite[Theorem 25]{bronnle15}}]
\label{thm:IFT}
Suppose $F : \cB_1 \to \cB_2$ is a differentiable map of Banach spaces whose derivative at $x_0$ is surjective with right-inverse $Q$. Let
\begin{itemize}
\item $\delta'$ denote the radius of a ball in $\cB_1$ such that $F - D_{x_0}F$ is Lipshitz with constant $\frac{1}{2\norm{Q}}$;
\item $\delta = \frac{\delta'}{2 \norm{Q}}$.
\end{itemize}
Then for all $z \in \cB_2$ such that $\| z - F(x_0) \|_{\cB_2} < \delta$ there exists an $x \in \cB_1$ with $\| x - x_0 \|_{\cB_1} < \delta'$ such that $F(x) = z.$
\end{theorem}

We can now prove Proposition \ref{prop:findimreduction}.
\begin{proof}[Proof of Proposition \ref{prop:findimreduction}]
We apply Theorem \ref{thm:IFT} to $\cB_1 = L^2_{q+4} \times \mfk$ and $\cB_2 = L^2_q$, with $F=\Phi_{k,v,j}$ and $x_0 = (0, \sigma_{k,v,j})$ for a sufficiently large $j$. We will let $v$ be sufficiently small and $k$ sufficiently large so that Proposition \ref{prop:approxsolnsSobolev}, Proposition \ref{prop:linearisation} and Lemma \ref{lem:mvt} apply. Note that this depends on $j$, but we will shortly fix which $j$ we will use.

To see that Theorem \ref{thm:IFT} gives a solution when $j$ is chosen sufficiently large, we first of all note that Lemma \ref{lem:mvt} implies that for all sufficiently small $\zeta$, $F - dF=\Phi_{k,v,j}-L_{k,v,j}$ is Lipschitz with Lipschitz constant $C \zeta$ on the ball of radius $\zeta$ in $L^2_{q+4} \times \mfk$. By Proposition \ref{prop:linearisation}, there is a $C'>0$ such that $\frac{1}{2\|Q_{k,v,j}\|} > C' k^{-3}$. Thus on the ball of radius $\delta'_k = \frac{C'}{Ck^3}$ in $L^2_{q+4} \times \mfk$, $\Phi_{k,v,j}-L_{k,v,j}$ is Lipschitz of Lipschitz constant $\frac{1}{2\|Q_{k,v,j}\|}$. Using the bound on $\frac{1}{2\|Q_{k,v,j}\|}$ again, the corresponding $\delta$ to which Theorem \ref{thm:IFT} applies is therefore bounded below by $\delta_k = C'' k^{-6}$ for some $C''>0$. 

The result follows if we then chose $j$ from the start so that Proposition \ref{prop:linearisation} applies, and such that Proposition \ref{prop:approxsolnsSobolev} gives an $O(k^{-7})$ bound for the $L^2_q(g_{k,v,j})$-norm of $\Phi_{k,v,j}(0,\sigma_{k,v,j})$ (which holds if $j \geq 6+\frac{n}{2}$).
\end{proof}


The above gives, for each $k$, a new Hermitian structure on the fibration $\cU \to \cV$. If these were relatively almost K\"ahler, we could apply the Dervan-Hallam Theorem \ref{thm:dervanhallam} to these metrics, to get new moment maps on $\cV$ for the $K$-action. However, as explained in Remark \ref{rem:smoothlyvarying}, the closedness of the resulting relatively Hermitian metrics on $\cU$ may fail, and so we need to argue differently to see that we obtain a moment map on $\cV$. For this, we use the argument of the first named author in \cite[Section 2]{Ortu24}.

Let $\phi_k$ be the function whose restriction to $\cU_v$ is $\phi_{k,v}$. Note that the $\phi_k$ are smooth. Further, let $\mathcal{M}_{\omega_k, \phi_k} : \cV \to \R$ be the function which, at a point $v\in \cV$, is the negative of the Mabuchi functional of the initial almost K\"ahler metric $\omega_k$ on $\cU_v$, evaluated at $\phi_k$. Then over the (smooth part of the) integrable locus, we have (\cite[Section 7]{PhongRossSturm08} and \cite[Proposition 4.2]{DervanNaumann18})
$$
\alpha_k - \beta_{k} = \ddb\left(\mathcal{M}_{\omega_k, \phi_k}\right),
$$
where $\beta_{k}$ is the Weil-Petersson form defined as in Definition \ref{Def:WeilPetersson} from the form $\omega_k + \ddb(\phi_k)$. As $\phi_k$ is a smooth function on the whole of $\cU$, $\mathcal{M}_{\omega_k, \phi_k}$ is a smooth function on $\cV$. We can then \emph{define} the new form $\beta_k$ on $\cV$ to be given by the above formula. 

The original form $\alpha_k$ is K\"ahler by Lemma \ref{lem:WPKahler} and we wish to show the same holds for $\beta_k$. Now, both $\phi_k$ and $\mrm{Scal}(\omega_k + i\del\delbar (\phi_{k,j}), J_{k,v,j})  - \hat S_k$ are $O(k^{-1})$, and $\omega_k$ is $O(k^{n})$. Thus, $\mathcal{M}_{\omega_k, \phi_k}$ is $O(k^{n-2})$. On the other hand, in Lemma \ref{lem:WPKahler}, we showed that $\alpha_k$ is bounded below by $k^{n-1}$ times a fixed positive K\"ahler form. As $\beta_k$ is an $O(k^{n-2})$ perturbation of $\alpha_k$, this shows that $\beta_k$ is K\"ahler on $\cV$ for all $k \gg 0$ as well, after possibly shrinking $\cV$. The moment map then also extends as for an element $\sigma \in \mfk$, it changes by $d^c \mathcal{M}_{\omega_k, \phi_k}(\eta_{\sigma})$, where $\eta_{\sigma}$ is the vector field on $\cV$ produced by $\sigma$ through the $K$-action. The upshot is that we have a K\"ahler form $\beta_{k}$ on $\cV$, which on the integrable locus is given by Theorem \ref{thm:dervanhallam} for the form on $\cU$ obtained by changing the form via the potential $\phi_k$, and a corresponding moment map $\mu_k$.

We end this section by showing that the zeros of the moment map we have now produced genuinely correspond to cscK metrics. We first explain what we have solved so far using the notation introduced in \S\ref{sec:mfk}. 

For every $v \in \cV$, we have produced a $\phi_{k,v} \in C^{\infty}(M)$ and an almost complex structure $J_{k,v}$ on $M$ such that 
$$
\Scal(\omega_k + \ddb \phi_{k,v}, J_{k,v}) - \hat S_k \in \overline \mfk^M_{\phi_{k,v}},
$$
where $J_{k,v}$ is the complex structure $J_{k,v,j}$ which we used when we proved Proposition \ref{prop:findimreduction}. By the smooth dependence on $v$ for the $\sigma_{k,v,j}$, the $J_{k,v}$ are the restrictions to the fibre over $v \in \cV$ of an almost complex structure $J_k$ on the underlying smooth manifold of $\cU$. We denote the corresponding complex manifold $\cU_k$ -- as a smooth manifold these are all isomorphic to $\cU$, i.e.\ to $M \times \cV$. The projection $\cU_k \to \cV$ is thus still a $K$-equivariant holomorphic map. 

The $\phi_{k,v}$ then glue to a function $\phi_k$ on $\cU_k$ which is actually smooth, again because of the smooth dependence of the $\phi_{k,v}$ on $v$. Thus we have shown that
$$
\Scal_{\mrm{}}((\omega_k + \ddb \phi_{k})\big|_{\cU_{k,v}}) - \hat S_k \in  \pr_2^* (\overline \mfk^M_{\phi_{k,v}})
$$
for every $v \in \cV$, where $\cU_{k,v}$ is the fibre at $v$ of $\cU_k \to \cV$.
However, the moment map of the $K$-action on $\cU_k$ takes values in the dual of $ \overline \mfk_{\phi_k} \big|_{\cU_{k,v}}$  when restricted to a fibre. If these were equal, we would immediately obtain that a zero of the moment map corresponds to a cscK metric on the fibre, but $\overline \mfk^M_{\phi_{k,v}} \neq \overline \mfk_{\phi_k} \big|_{\cU_{k,v}}$ in general.

\begin{definition}\label{Def:momentmap_muk}
Let $\tau_k : \cU_k \to \mfk^*$ be the sequence of moment maps for the $K$-action on $\cU_k$ with respect to $\omega_k + \ddb \phi_k$. We denote by $\mu_k$  the corresponding sequence of moment maps on $\cV$, given by Theorem \ref{thm:dervanhallam}.
\end{definition}

\begin{lemma}\label{Lemma:injective_projection}
A point $v \in \cV^{\mrm{int}}$ is a zero of the moment map $\mu_k$ if and only if the scalar curvature of $\omega_k$ on $\P(\cE_v)$ is constant.
\end{lemma}
\begin{proof}
Let $\iota_{1,k,v}, \iota_{2,k,v} : \mfk \to C^{\infty} (M)$ be given by
$$
\iota_{1,k,v} (\sigma) = \left( h_{\sigma} + \nu_{\sigma}(\phi_k) \right)\big|_{\cU_v}
$$
and
$$
\iota_{2,k,v} (\sigma) = \widetilde h_{\sigma} + \xi_{\sigma}(\phi_{k,v}) .
$$
The $\iota_{j,k,v}$ are injective linear maps. The corresponding images are then isomorphic to $\mfk$. We then have that $\iota_{1,k,v}(\mfk) = \overline \mfk_{\phi_k}\big|_{\cU_v},$ which is the image of the moment map $\mu_{\phi_k}$ with respect to $\omega_k + \ddb \phi_k$ associated to the $K$-action on $\cU$, thought of as a map from $\mfk$ to $C^{\infty}(\cU)$, restricted to the fibre $\cU_v$. Moreover, $\iota_{2,k,v}(\mfk) = \overline \mfk^M_{\phi_{k,v}}$.

Arguing as in the proof of Lemma \ref{lem:momentmapzero}, we see that $v \in \cV$ is a zero of the moment map if and only if $\Scal_V(\omega_k+\ddb\phi_{k,v}, J_v) - \hat S_k$ is orthogonal to the image of $\mu_{\phi_k}$, i.e. to $\iota_{1,k,v}(\mfk)$. Let $\Pi_{k,v}$ denote the $L^2(g_{k,v})$-orthogonal projection to $\iota_{1,k,v}(\mfk)$. Note that we have ensured that $\Scal(\omega_k+\ddb\phi_{k,v}, J_v) - \hat S_k \in \iota_{2,k,v}(\mfk)$ for all $v$ and $k$. The result will now follow from the claim that $\Pi_{k,v}$ restricted to $\iota_{2,k,v}(\mfk)$ is an isomorphism, as this then implies that $\Scal(\omega_k+\ddb\phi_{k,v}, J_v) - \hat S_k = 0$ if and only if $\Pi_{k,v} (\Scal(\omega_k+\ddb\phi_{k,v}, J_v) - \hat S_k) =0$. 

To see this, first note that
\begin{align*}
\iota_{1,k,v} (\sigma) =&\iota_{2,k,v} (\sigma) + \eta_{\sigma}(\phi_k)\big|_{\cU_v}
\end{align*}
since the restriction of $h_{\sigma}$ to $\cU_v$ is $\widetilde h_{\sigma}$ and the vertical part of $\nu_{\sigma}$ is $\xi_{\sigma}$. Now, let $\sigma^1, \ldots, \sigma^d$ denote a basis of $\mfk$ such that the $\widetilde h_i = \widetilde h_{\sigma^i}$ form an orthonormal basis with respect to the volume form $\omega^r \wedge \omega_B^n$.

We regard the restriction of $\Pi_{k,v}$ to $\iota_{2,k,v}(\mfk)$ as a linear map from $\R^d$ to itself, through the isomorphisms $\iota_{j,k,v} : \mfk \to \iota_{j,k,v}(\mfk)$ and by using the above basis for $\mfk$.  The map is then 
$$
(\lambda_1, \ldots, \lambda_d) \mapsto (c_1, \ldots, c_d),
$$
where
\begin{align*}
c_j =& \sum_{i=1}^d \lambda_i \left( \frac{\int_M \iota_{2,k,v} (\sigma_i) \cdot \iota_{1,k,v} (\sigma_j)  (\omega_k+\ddb \phi_{k,v})^{r+n}}{\|\iota_{1,k,v} (\sigma_j)  \|_{L^2(\omega_{k}+\ddb \phi_{k,v})} } \right) .
\end{align*}
Now, as $\eta_{\sigma}$ vanishes when $v=0$ and $\phi_k = O(|v|)$, we see that 
$$
 \iota_{2,k,v} (\sigma_i) \cdot \iota_{1,k,v} (\sigma_j) =  \iota_{2,k,v} (\sigma_i) \cdot \iota_{2,k,v} (\sigma_j)  + O(|v|^2).
$$
Similarly, the contribution from $\xi_{\sigma}$ is $O(|v|)$ (in fact, $O(|v|k^{-1})$ since $\xi_{\sigma}$ is vertical with respect to the fibration structure $M \to B$). So, we see that
$$
 \iota_{2,k,v} (\sigma_i) \cdot \iota_{1,k,v} (\sigma_j) = \widetilde h_i \cdot \widetilde h_j + O(|v|).
$$
Expanding the volume forms and using the above, we then see that
\begin{align*}
\frac{\int_M \iota_{2,k,v} (\sigma_i) \cdot \iota_{1,k,v} (\sigma_j)  (\omega_k+\ddb \phi_{k,v})^{r+n}}{\|\iota_{1,k,v} (\sigma_j)  \|_{L^2(\omega_{k}+\ddb \phi_{k,v})} } = \int_M \widetilde h_i \cdot \widetilde h_j \omega^r \wedge \omega_B^n + O(|v|).
\end{align*}
Since the $\widetilde h_i$'s form an orthonormal basis with respect to $\omega^r \wedge \omega_B^n$, the upshot is that
\begin{align*}
c_j =& \sum_{i=1}^d \delta_{j}^i \lambda_i + O(|v|) ,
\end{align*}
which is a perturbation of the identity map $\mfk \to \mfk$. Thus for all sufficiently small $v$, $\Pi_{k,v}$ is an isomorphism as well, giving the result.
\end{proof}

\begin{remark} Lemma \ref{Lemma:injective_projection} can also be applied to resolve an analogous issue about the discrepancy in the projection onto the different Lie algebra that appears in an argument of Dervan \cite{dervan23}, in the proof of their Corollary 3.49. We thank R.\ Dervan for pointing out this issue to us. This method should apply generally to semistable perturbation problems to allow one to solve an easier fibrewise problem instead of a global problem on the (non-compact) universal family.
\end{remark}

\section{Adiabatic slope stability and the moment map flow}
\label{sec:findimsoln}
To construct constant scalar curvature metrics on the total space we use the flow associated with the sequence of moment maps of Definition \ref{Def:momentmap_muk}, following \cite[\S4.2]{DervanMcCarthySektnan}, which builds on \cite{ChenSun_CalabiFlow,GeorgoulasRobbinSalamon_GITbook}.
Let $\m{E} \to B$ be a semistable vector bundle, let $v \in \m{V}$ be the corresponding point in the Kuranishi space and let us denote by $(\omega_{k,v}, J_{k,v})$ the K\"ahler metric on $\bb{P}(\m{E}_v) \to B$ obtained from Proposition \ref{prop:findimreduction}.
Let $\mu_{k}$ be the associated moment on $\m{V}$ of Definition \ref{Def:momentmap_muk}. For each $k$, the moment map flow associated to the moment map $\mu_k$ is given by 
\[
\begin{aligned}
&\frac{\dd}{\dd t} v_t = J \eta({\mu_{k}(v_t)})\\
&v_{0} = v,
\end{aligned}
\]
where $\eta({\mu_{k}(v_t)})$ is the infinitesimal vector field induced by $\mu_k(v_t)$.
This is the gradient flow of the norm squared of the moment map \cite[Chapter 3]{GeorgoulasRobbinSalamon_GITbook}.
Up to shrinking $\m{V}$, since the origin of $\m{V}$ is in the closure of the orbit of $v$, for each $k$ the moment map flow starting at a point $v$ converges to a unique point $v_{k,\infty}$, which lies in $\m{V}$ \cite[\S4.2]{DervanMcCarthySektnan}.
Moreover, since $v$ represents an integrable complex structure and integrability is a closed condition, the limit point is also integrable.

We need the following characterisation \cite[Corollary 4.14]{DervanMcCarthySektnan}.
\begin{proposition}\label{Prop:zero_in_orbit}
Let $\mu$ be a moment map on $\m{V}$ with respect to a K\"ahler form $\alpha$.
Let $x \in \m{V}$ with finite stabiliser such that $0 \in \overline{K^{\bb{C}}\cdot x}$ and let $x_{\infty}$ be the limit of the moment map flow starting at $x$. Then either one of the following happens:
\begin{enumerate}
\item\label{Item:inorbit} $\mu(x_{\infty}) = 0$ and $x_{\infty} \in K^{\bb{C}}\cdot x$; 
\item\label{Item:notinorbit} $x_{\infty} \notin K^{\bb{C}}\cdot x$ and there exists an element $\eta \in \mf{k}$ and a point $\widetilde{x}$ such that
\[
\lim_{t \to \infty} \exp(-it\eta) \cdot x = \widetilde{x}
\]
and $\langle \mu (\widetilde{x}), \eta \rangle\ge 0$.
\end{enumerate}
\end{proposition}
We apply Proposition \ref{Prop:zero_in_orbit} to the sequence of moment maps $\mu_k$ on $\m{V}$, induced by the sequence of K\"ahler forms $\omega_k$ on the projectivised vector bundle.
For each $k$, let $v_{k, \infty}$ be the limit point of the moment map flow starting at $v$. Let $\eta_{k, \infty}$ be the corresponding vector field in $\mf{k}$, i.e.\
\[
\lim_{t \to 0} \exp(-it\eta_{k,\infty})\cdot v = v_{k, \infty}.
\]
Let $\m{E}_{k,\infty}$ be the vector bundle with holomorphic structure $\delbar_{k,\infty}$ induced by the deformation $\widetilde\Phi(v_{k,\infty})$, where $\widetilde \Phi$ is the Kuranishi map \eqref{eq:Kuranishimap_vb}.
Then $\eta_{k,\infty}$ induces a test configuration for $(\bb{P}(\m{E}_v), \omega_k)$ with central fibre $\bb{P}(\m{E}_{\widetilde{v}_k})$, as explained in Remark \ref{Rmk:explicit_testconf}. We denote it by $\m{X}_{\eta_{k,\infty}}$.

We next describe how $\eta_{k,\infty}$ induces a filtration for $\m{E}_v \to B$ following \cite[\S 4.2.3]{DervanMcCarthySektnan}. For each $\lambda \in \bb{R}$, denoted by $\sigma_{k,\infty} \in \mf{k}$ the endomorphism of $\m{E}_v \to B$ corresponding to $\eta_{k,\infty}$,
\[
\m{P}_\lambda = \ker(i \sigma_{k,\infty}-\lambda \mrm{Id}_{\m{E}_v})
\]
is a holomorphic subbundle of $\m{E}_{k, \infty}$.
In particular, we split $\sigma_{k, \infty} = \sum_{i=1}^d \lambda_i \sigma_{\lambda_i}$, where $\lambda_i$ are eigenvalues of $\sigma_{k, \infty}$ such that
\[
\lambda_1 > \ldots > \lambda_d.
\]
Note that $d \geq 2$, as otherwise $\m{E}_{k, \infty} \cong \m{E}_v$ and $\sigma_{k,\infty}$ is then the identity, since $\m{E}_v$ is simple. We can write $\m{E}_{k, \infty}$ as the direct sum of the $\m{P}_i = \m{P}_{\lambda_i}$, and there are at least two such components.
The $\m{P}_i$'s induce in turn a filtration of $\m{E}_v$ by holomorphic subbundles \cite[Corollary 4.18]{DervanMcCarthySektnan}
\[
0 \subset \m{S}_1 \subset \dots \subset \m{S}_d = \m{E}_v
\]
where $\m{S}_j = \oplus_{i\le j} \m{P}_{i}$.
Therefore, each sub-filtration $0 \subset \m{S}_j \subset \m{E}_v$ induces a test configuration $(\m{X}_j, \m{H}+ kL)$ as explained in Remark \ref{Rmk:explicit_testconf}, via the vector field $\eta_i = \eta_{\sigma_{\lambda_i}}$.

\begin{lemma}
\label{lem:DFexpansion}
The pairing of the moment map $\mu_{k}$ with the vector field $\eta_{k,\infty}$ is given by
\[
\langle\mu_k(v_{k,\infty}), \eta_{k,\infty} \rangle = - \frac{\mrm{DF}(\m{X}_{\eta_{k,\infty}}, \m{H}+ kL)}{(n+r)!\pi}
\]
where $\m{H} = \m{O}_{\bb{P}(\m{E}_v)}(1)$. Furthermore, 
\[
\begin{split}
\mrm{DF}(\m{X}_{\eta_{k,\infty}}, \m{H}+ kL) = (\lambda_{d-1}-\lambda_{d}) \mrm{DF}(\m{X}_d, \m{H}+ kL) +&(\lambda_{d-2}-\lambda_{d-1}) \mrm{DF}(\m{X}_{d-1}, \m{H}+ kL)+\\ & \dots + (\lambda_1-\lambda_{2}) \mrm{DF}(\m{X}_1, \m{H}+ kL).
\end{split}
\]
\end{lemma}
\begin{proof}
Since the test configurations are smooth, from Theorem \ref{Thm:classical_Futaki_equal_DF} the Donaldson-Futaki invariant is equal to the negative of the classical Futaki invariant of the central fibre, computed on the vector field inducing the test configuration:
\[
-\frac{\mrm{DF}(\m{X}_{\eta_{k,\infty}}, \m{H}+ kL)}{(n+r)!\pi} = \mrm{Fut}(\eta_{k,\infty}) = \langle\mu_k(v_{k,\infty}), \eta_{k,\infty} \rangle.
\]
Moreover, since $\Id_{P_i} = \Id_{S_i} - \Id_{S_{i-1}}$, we can write 
\[
\sigma_{k,\infty} = \sum_{i=1}^d \lambda_i \mrm{Id}_{P_{i}} = \lambda_d \Id_{S_d} + \sum_{i=1}^{d-1} (\lambda_i-\lambda_{i+1}) \mrm{Id}_{S_i},
\]
so in particular, since $\m{S}_d = \cE_v$ and $\eta_{d} = \eta_{\Id_{\cE_v}} = 0$ as the identity on $\dbar_v$ acts trivially,
\[
\eta_{k,\infty} =  \sum_{i=1}^{d-1} (\lambda_i-\lambda_{i+1}) \eta_{i}.
\]
Therefore, 
\begin{align*}
\frac{\mrm{DF}(\m{X}_{\m{F}}, \m{H}+ kL)}{(n+r)! } =& -\pi \Fut(\eta_{k,\infty}) \\
=&-\pi \sum_{i=1}^{d-1} (\lambda_i-\lambda_{i+1}) \Fut(\eta_{\lambda_i} ) \\
=& \sum_{i=1}^{d-1} (\lambda_i-\lambda_{i+1}) \frac{\mrm{DF}(\m{X}_i, \m{H}+ kL)}{(n+r)! },
\end{align*}
where each coefficient $\lambda_i -\lambda_{i+1}$ is positive.
\end{proof}

\begin{theorem}\label{Thm:stability}
Let $(\bb{P}(\m{E}_v), \omega_k) \to B$ be adiabatically slope stable. Then there is $v_{k,\infty} \in \m{V}$ in the $K^{\bb{C}}$-orbit of $v$ such that
\[
\mrm{Scal}(\omega_k,J_{v_{k,\infty}}) = constant
\]
for every $k \gg 0$.
\end{theorem}

\begin{proof}
Assume that condition \eqref{Item:notinorbit} of Proposition \ref{Prop:zero_in_orbit} occurs.  
We can therefore find $\widetilde{v}_k \in \cV$ and $\eta_k \in \mf{k}$ such that $\lim_{t \to \infty} \exp(-i t \eta_k) \cdot v =  \widetilde{v}_k$ and
\[
\langle \mu_{k,v} (\widetilde{v}_k), \eta_k\rangle\ge 0.
\]
Then $\eta_{k}$ induces a test configuration for $(\bb{P}(\m{E}_v), \omega_k)$ with central fibre $\bb{P}(\m{E}_{\widetilde{v}_k})$.
By Lemma \ref{lem:DFexpansion}, the pairing
\[
\langle \mu_{k,v}(\widetilde{v}_k), \eta_k\rangle
\]
is equal to the negative of the Donaldson-Futaki invariant of said test configuration. By assumption $\bb{P}(\m{E}_v) \to B$ is adiabatically slope stable, so using Lemma \ref{lem:DFexpansion} again, we see that 
\[
\langle \mu_{k,v}(\widetilde{v}_k), \eta_k \rangle < 0
\]
unless the test configuration induced by $\eta_k$ is a product test configuration. But $\widetilde{v}_k$ is not in the orbit of $v$, and since $\widetilde{v}_k$ has non-trivial stabiliser, but $v$ corresponds to a simple bundle, $J_{\widetilde{v}_k}$ is not biholomorphic to $J_v$. So the test configuration is not a product test configuration.

Hence from Proposition \ref{Prop:zero_in_orbit} we obtain that the condition \eqref{Item:inorbit} must hold, so there is a $v_{k,\infty} \in K^{\bb{C}}\cdot v$ such that $\mu_{k,v}(v_{k,\infty}) = 0$ and $\Phi(v_{k,\infty})$ is an integrable complex structure.
From Lemma \ref{Lemma:injective_projection}, $v_{k, \infty}$ is such that the scalar curvature of the K\"ahler metric on $\bb{P}(\m{E}_{v_{k, \infty}})$ is constant.
\end{proof}

We deduce the following corollary.
\begin{corollary}
With the same assumptions on $\m{E}$ and $B$, suppose $\bb{P}(\m{E}) \to B$ is K-stable with respect to $H+kL$ for all sufficiently large $k$. Then $\bb{P}(\m{E})$ admits a cscK metric in $c_1(H+kL)$ for all sufficiently large $k$.
\end{corollary}

While our methods are analytical, we obtain the following algebro-geometric statement.
\begin{corollary}
Suppose $\cE$ is a sufficiently smooth, adiabatically slope stable vector bundle over a base $B$ with discrete automorphism group admitting a cscK metric in $c_1(L)$. Then $(\P(\cE), H+kL)$ is K-stable for all sufficiently large $k$.
\end{corollary}

We also have that for sufficiently smooth vector bundles, it suffices to check adiabatic slope stability on subbundles. 
\begin{corollary}
\label{cor:adiabaticsubbundles}
Suppose $\cE$ is a sufficiently smooth vector bundle over a base $B$ with a discrete automorphism group admitting a cscK metric in $c_1(L)$, which is adiabatically slope stable with respect to subbundles. Then $\cE$ is adiabatically slope stable.
\end{corollary}

\section{Asymptotic expansion of the Donaldson-Futaki invariant}
\label{sec:example}
In Section \ref{sec:findimsoln} we have seen that to check adiabatic slope stability it is enough to check on the test configurations induced by subbundles.
In this section, we write an asymptotic adiabatic expansion of the Donaldson-Futaki invariant of said test configuration, and we exhibit an example of a bundle satisfying our assumptions.

\subsection{Asymptotic expansion in general}

In this section, we will give a generating formula for the asymptotic expansion of the Donaldson-Futaki invariant of the test configuration $(\cX_{\m{S}}, \cH+kL)$ considered in \S\ref{sec:adiabaticstability} when $\m{S} \subset \cE$ is a subbundle. In order to obtain the formula, we will use Legendre's localisation formula for the Donaldson-Futaki invariant \cite{LegendreLocalisation21}.

First of all, note that $(\cX_{\m{S}}, \cH+kL)$ is obtained by flowing along the vector field generated by $m(\Id_{S})$. Indeed, if we let $\m{Q}=\frac{\cE}{\m{S}}$ and $\dbar_0$ be the holomorphic structure of $\m{S} \oplus \m{Q}$, then the holomorphic structure on $\cE$ is given by $\dbar = \dbar_0 + \gamma$, where $\gamma \in \Omega^{0,1}(Q^* \otimes S)$. It follows that 
$$
\exp(t \Id_{S}) \cdot \dbar = \dbar_0 + e^{-t} \gamma
$$
and so
\begin{align*}
\lim_{t \to \infty} \exp(t \Id_{S}) \dbar = \dbar_0.
\end{align*}
Thus $\cE$ degenerates to $\m{S} \oplus \m{Q}$ via $\Id_{S}$ and therefore the projectivisation $\mathbb{P}(\cE)$ degenerates to $\mathbb{P}(\m{S} \oplus \m{Q})$ via the vector field generated by $m(\Id_{S})$.

By Theorem \ref{Thm:classical_Futaki_equal_DF}, we can therefore compute the Donaldson-Futaki invariant of $(\cX_{\m{S}}, \cH+kL)$ via the Futaki invariant on the central fibre, which is $\mathbb{P}(\m{S} \oplus \m{Q})$. This is where we apply the localisation formula of Legendre, which we now describe. 

Let $(Y, \omega_Y)$ be a compact K\"ahler manifold of dimension $m$.
Suppose $f$ is a Hamiltonian for an $S^1$-action generated by a holomorphic vector field $\nu \in \mf{k}$. Let $Z_1, \ldots, Z_l$ be the connected components of the fixed point set of this $S^1$-action: on each such component, the restriction $f_i = f\big|_{Z_i}$ is a constant.

Let $N_i$ be the normal bundle of $Z_i$ in $Y$, for each $i$; then $\omega_Y$ induces a Hermitian metric on $N_i$, with curvature $F_i$.

Let $\Delta = \dd^*\dd$ be the Laplacian on functions on $Y$, so that on each $Z_i$
\begin{equation}\label{Eq:EvelinesLaplacian_weights}
\Delta f = -2 \lambda_i f
\end{equation}
where the constant $\lambda_i$ on $Z_i$ is the weight of the action of $K$ on the normal bundle $N_i$.

\begin{definition}
The equivariant Euler class of $N_i$ is a $\mrm{rk} N_i$-equivariant differential form on $Y$.
Its Chern-Weil representative at $\nu$ is
\[
e_{eq}(N_i)_\nu = P_{\mrm{rk}N_i} (F_i + f),
\]
where $P_{\mrm{rk}N_i}$ is the $\mrm{rk}N_i$-th homogeneous component of $\det (\mrm{id} + F_i)$.
\end{definition}
Thought of as represented by an element in the cohomology ring of $Z$, the equivariant Euler class has a non-zero degree $0$ component, and can thus be inverted -- when we use division notation below, we mean multiplication by this inverse.

The form in which we will use the localisation formula is given by the following.
\begin{proposition}[{\cite[Equation (9)]{LegendreLocalisation21}}]
\label{thm:FutLoc}
With the notation above, 
$$
\begin{multlined}
\Fut(\nu) =  \sum_{i=1}^l \frac{m}{(m+1)!} \widehat S_Y\left\langle\frac{([\omega_Y]-f_1)^{m+1}}{e_{eq}(N_i)}, [Z_i] \right\rangle -\sum_{i=1}^l \frac{1}{m!} \left\langle \frac{\left(c_1(Y) + \lambda_if_i \right).([\omega_Y] - f_{i})^m}{e_{eq}(N_i)} , [Z_i] \right\rangle
\end{multlined}
$$
where $\widehat S_Y$ is the average scalar curvature and with the brackets we denote the intersection product computed on the class of the subvariety $Z_i$.
\end{proposition}

We compute the localisation formula for the action generated by $f= m(\Id_{S})$ with $Y=\mathbb{P}(\m{S}\oplus \m{Q})$ and $\omega_Y = \omega_k=\omega + k \omega_B$.
With this choice of $f$, the fixed point set is $Z_1 \cup Z_2$, where $Z_1 = \mathbb{P}(\m{S})$ and $Z_2 = \mathbb{P}(\m{Q}).$ Moreover
$$
m(\Id_{S}) =
\begin{cases}
-1 & \textnormal{on } \mathbb{P}(\m{S})\\
0 & \textnormal{on } \mathbb{P}(\m{Q})
\end{cases}
$$
We start by multiplying the Futaki invariant by the volume of $\omega_k$, as this makes the expression a polynomial in $k$, rather than a rational function. Note that this does not change its sign.
Since $f$ vanishes on $\bb{P}(\m{Q})$, the only addendum in the Futaki invariant is the one evaluated on $\bb{P}(\m{S})$. Therefore, the quantity we compute is
\[
\begin{multlined}
[\omega_k]^{n+r}\Fut(\nu) =  \frac{n+r}{(n+r+1)!} \left\langle c_1(X) \cdot [\omega_k]^{n+r-1}, X \right\rangle\left\langle\frac{([\omega_k]+1)^{n+r+1}}{e_{eq}(N_{\m{S}})}, [\bb{P}(\m{S})] \right\rangle +\\
 -\frac{1}{(n+r)!}\langle[\omega_k]^{n+r}, X \rangle \left\langle \frac{\left(c_1(X) - \lambda_{\m{S}} \right).([\omega_k] + 1)^{n+r}}{e_{eq}(N_{\m{S}})} , [\bb{P}(\m{S})] \right\rangle.
\end{multlined}
\]
Here we continue using the bracket notation to denote the object on which we compute the intersection product.
We denote
\[
q = \rk\m{Q}, \quad s+1 = \rk\m{S}.
\]

We first compute the constant $\lambda_{\m{S}}$.
The eigenfunctions of the Laplacian with respect to the metric induced by the canonical polarisation $O_{\bb{P}^{r}}(r+1)$ with eigenvalue $2$ \cite[Lemma 3.16]{dervansektnan21a} are induced by a traceless element of $\mf{k}$, such as the function
\[
m\left(\mrm{Id}_S-\frac{\rk \m{S}}{\rk \m{E}}\mrm{Id}_{\m{E}}\right).
\]
However, we consider the polarisation $O_{\bb{P}^{r}}(1)$ instead of the canonical polarisation, and
\[
c_1(O(r+1)) = (r+1)c_1(O(1)).
\]
Therefore from equation \eqref{Eq:EvelinesLaplacian_weights} it follows that on $\bb{P}(\m{S})$,
\[
\begin{aligned}
\Delta (m(\Id_S)) =& \Delta \left( m\left(\mrm{Id}_S-\frac{\rk \m{S}}{\rk \m{E}}\mrm{Id}_{\m{E}} \right)\right)\\
=& -2 (r+1) m\left(\mrm{Id}_S-\frac{\rk \m{S}}{\rk \m{E}}\mrm{Id}_{\m{E}} \right) \\
=&-2 (r+1) \left( \left( 1-\frac{\rk \m{S}}{\rk \m{E}}\right) m(\Id_S) + \frac{\rk S}{\rk \m{E}} m (\Id_Q)\right)\\
=& -2 (r+1)\left( 1-\frac{\rk \m{S}}{\rk \m{E}}\right) m(\Id_S),
\end{aligned}
\]
where we have used that $m(\Id_Q)$ vanishes on $\bb{P}(\m{S})$. Hence
\[
\lambda_S =-\rk \m{Q}.
\] 

Next, to apply the formula we need to compute the equivariant Euler class of the normal bundle. The normal bundle of $\mathbb{P}(\m{S})$ in $\mathbb{P}(\m{S} \oplus \m{Q})$ is given by 
$$
N_{\m{S}} = H \otimes \pi^*\left(\m{Q} \right),
$$
see \cite[Proposition 9.13]{3264}.
By the equivariant splitting principle \cite[Theorem 8.6.2]{GuilleminSternberg99}, we can write $\pi^*\left(\m{Q} \right)$ as sum of line bundles
\[
\pi^*\left(\m{Q} \right) = \bigoplus_{i=1}^q \m{L}_i.
\]
The Hermitian metric on $\m{S} \oplus \m{Q}$ induces a Chern connection $\nabla_1$ which is equivariant with respect to the $S^1$-action on the normal bundle $N_1$. The equivariant Euler class of $N_1$, which equals the top equivariant Chern class of the bundle since it is a complex vector bundle, is then represented by
\[
\begin{split}
e_{eq}(N_{\m{S}}) &= \prod_{i=1}^q \left( c_1 (H \otimes \m{L}_i) + 1\right)\\
&=1 + \sum_{i=1}^q c_1(H \otimes \m{L}_i) + \sum_{i, j=1, i< j}^q c_1(H \otimes \m{L}_i)c_1(H \otimes \m{L}_j)  +\dots + \prod_{i=1}^qc_1(H \otimes \m{L}_i) \\
&=1 + c_1(H \otimes \pi^*\m{Q}) + c_2(H \otimes \pi^*\m{Q}) + \dots + c_q(H \otimes \pi^*\m{Q}) \\
&= c(H \otimes \pi^*\m{Q}).
\end{split}
\]
Its inverse is the total Segre class \cite[Chapter 10]{3264}:
\[
e_{eq}(N_{\m{S}})^{-1} = s(N_{\m{S}}).
\]

Finally, we need to compute the first Chern class of $X$. From the relative Euler sequence, and since $\cE$ has rank $r+1$ and $c_1(\cE)=c_1(\m{S}\oplus \m{Q})$, one sees that 
\begin{align*}
c_1(X) =& c_1(B)+c_1(H \otimes \pi^*\cE) \\
=& c_1(B) + (r+1) c_1(H) + c_1(\cE).
\end{align*}

Putting everything together
\begin{equation}\label{Eq:Fut_expression}
\begin{aligned}
\begin{multlined}
[\omega_k]^{n+r}\mrm{Fut} (\sigma) = \frac{n+r}{(n+r+1)!}\left\langle c_1(X) \cdot [\omega_k]^{n+r-1}, X\right\rangle\left\langle([\omega_k]+1)^{n+r+1}. s(H \otimes \pi^*\m{Q}), [\bb{P}(\m{S})] \right\rangle
\end{multlined}\\
-\frac{1}{(n+r)!}\left\langle[\omega_k]^{n+r}, X\right\rangle \left\langle\left(c_1(X) + q\right). ([\omega_k]+1)^{n+r} . s(H \otimes \pi^*\m{Q}), [\bb{P}(\m{S})] \right\rangle
\end{aligned}
\end{equation}

\subsection{Example}
We will now compute the first three terms in the $k$-expansion of the Futaki invariant \eqref{Eq:Fut_expression} explicitly under the simplifying assumption that $q=1$.  We will frequently use the following result \cite{diverio2016segre, Guler2012}.
\begin{lemma}\label{Lemma:diverio}
Let $\bb{P}(\m{E}) \to B$ be the projectivisation of a vector bundle of rank $r+1$. Let $\omega$ be a relatively Fubini-Study metric on $\bb{P}(\m{E})$ and $\beta$ be a cholomology class of degree $2j$ on $B$. Then
\[
\int_{\bb{P}(\m{E})} \widetilde \beta \wedge \omega^{r+j} = \beta \cdot s_j(\m{E}),
\]
where $\widetilde \beta$ is a representative of $\beta$ and $s_j(\m{E})$ is the $j$-th Segre class of $\m{E}$.
\end{lemma}

We denote
\begin{align*}
\left\langle[\omega_k]^{n+r}, X\right\rangle =& \sum_{i=0}^n \alpha_i k^{n-i} \\
\left\langle c_1(X) . [\omega_k]^{n+r-1}, X\right\rangle =& \sum_{i=0}^n \beta_i k^{n-i} \\
\left\langle([\omega_k]+1)^{n+r+1}.s(H \otimes \pi^*\m{Q}), [\bb{P}(\m{S})] \right\rangle =& \sum_{i=0}^n \gamma_i k^{n-i} \\
 \left\langle\left(c_1(X) +q \right). ([\omega_k]+1)^{n+r} .s(H \otimes \pi^*\m{Q}), [\bb{P}(\m{S})] \right\rangle=& \sum_{i=0}^n \delta_i k^{n-i}.
\end{align*}
Thus the expansion of the Futaki invariant is
\begin{align*}
[\omega_k]^{n+r}\mrm{Fut} (\sigma) 
=&  \left( \frac{n+r}{(n+r+1)!} \beta_0 \gamma_0 - \frac{1}{(n+r)!}\alpha_0 \delta_0 \right) k^{2n} \\
&+ \left( \frac{n+r}{(n+r+1)!}  (\beta_0 \gamma_1 + \beta_1 \gamma_0)- \frac{1}{(n+r)!}(\alpha_0 \delta_1 + \alpha_1 \delta_0)\right) k^{2n-1} \\
&+ \left( \frac{n+r}{(n+r+1)!}  (\beta_0 \gamma_2 + \beta_1 \gamma_1 + \beta_2 \gamma_0) - \frac{1}{(n+r)!}(\alpha_0 \delta_2 + \alpha_1 \delta_1 + \alpha_2 \delta_0)\right) k^{2n-2} \\
&+O(k^{2n-3}).
\end{align*}

First we compute the expansion of the volume, i.e.\ the $\alpha$-terms.
Using that $L^{n+1} = 0$, Lemma \ref{Lemma:diverio} and the relations $s_1(\m{E}) = -c_1(\cE)$ and that $s_2(\cE) = c_1(\cE)^2 - c_2(\cE)$ one sees that
\begin{align*}
\alpha_0=& {n+r \choose n} L^n \\
\alpha_1=& - {n+r \choose n-1} L^{n-1}.c_1(\cE) \\
\alpha_2=& {n+r \choose n-2}  L^{n-2}.\left( c_1(\cE)^2 - c_2(\cE)\right).
\end{align*}

Similarly, by expanding $c_1(X) .[\omega_k]^{n+r-1}$, we see that
\begin{align*}
\beta_0=& {n+r-1 \choose n} \cdot (r+1) L^n \\
\beta_1=& {n+r-1 \choose n-1} \cdot \left( c_1(B) - r c_1 (\cE) \right) L^{n-1} \\
\beta_2=&{n+r-1 \choose n-2} \cdot \left(-c_1(B) c_1(\cE) + r c_1(\cE)^2- (r+1)c_2(\cE) \right)L^{n-2}.
\end{align*}

Next, we consider the $\gamma$-terms, i.e.\ we compute
$$
\left\langle([\omega_k]+1)^{n+r+1}\cdot (s(H \otimes \pi^*\m{Q})), [\bb{P}(\m{S})] \right\rangle.
$$
We have that if $V=H \otimes \pi^*\m{Q}$, then, since this is a line bundle, we have that $c_i(V) = 0$ for $i>1$, and so one can inductively prove that
$$
s_i (V) = (-1)^i c_1(V)^i.
$$
We therefore get that
\begin{align*}
s_i(V) =& (-1)^i c_1(V)^i \\
=&(-1)^i \left(H+\pi^*\m{Q}\right)^i.
\end{align*}
Note also that $\m{Q}.L^n = 0$, and so 
$$
s_i(V).L^n = (-1)^i H^i.L^n
$$
and in particular
$$
s(V).L^n = \sum_{i=0}^n H^i.L^n.
$$
Similarly, $\m{Q}^2.L^{n-1} = 0$, and so 
\begin{align}
\label{eq:Ln-1dot}
s_i(V).L^{n-1} = (-1)^i (H^i+iH^{i-1}.\m{Q}).L^{n-1}.
\end{align}
Note also that since the dimension of $\mathbb{P}(S)$ is $n+r-1$ and we are regarding $H \otimes \pi^* \m{Q}$ as a bundle over $\mathbb{P}(\m{S})$, the final term in the total Segre class $s(V)$ is the term of degree $n+r-1$.

With this place, we then see that
\begin{align*}
\gamma_0=&  {n+r+1 \choose n}  \left\langle(1+H)^{r+1}L^n .s (V), [\bb{P}(\m{S})] \right\rangle \\
=& {n+r+1 \choose n}  \left\langle(1+H)^{r+1} \sum_{i=0}^{n+r-1} (-1)^i H^i.L^n, [\bb{P}(\m{S})] \right\rangle \\
=& {n+r+1 \choose n} \sum_{j=0}^{r+1} \sum_{i=0}^{n+r-1}  (-1)^i {r+1 \choose j}  \left\langle H^{i+j}.L^n, [\bb{P}(\m{S})] \right\rangle .
\end{align*}

Now, this term is only non-zero when $i+j=r-1$, since $\bb{P}(\m{S})$ has dimension $n+r-1$. We therefore have that for each $j$, we need to pick $i=r-1-j$, but we also need to ensure that $i \in [0, n+r-1]$. The inequality $i\geq0$ then gives that $j\leq r-1$. The inequality $i \leq n+r-1$ gives that $j \geq r-1-(n+r-1)=-n$, which is always satisfied as $j \geq 0$. Thus 
\begin{align*}
\gamma_0=&{n+r+1 \choose n} \sum_{j=0}^{r-1} (-1)^{r-1-j} {r+1 \choose j}  \left\langle H^{r-1}.L^n, [\bb{P}(\m{S})] \right\rangle \\
=& {n+r+1 \choose n} \sum_{j=0}^{r-1} (-1)^{r-1-j} {r+1 \choose j} L^n .
\end{align*}

We treat the terms $\gamma_1$ and $\gamma_2$ similarly. By using Equation \eqref{eq:Ln-1dot} for the $\gamma_1$-term and the analogous equation 
\begin{align*}
s(V).L^{n-2} = \sum_{i=0}^{n+r-1} (-1)^i (H^i+iH^{i-1}.\m{Q} + \frac{1}{2}i (i-1)H^{i-2}.\m{Q}^2).L^{n-2}.
\end{align*}
coming from $\m{Q}^3.L^{n-2} = 0$ for the $\gamma_2$-term, one can by proceeding as above  verify that
\begin{align*}
\gamma_0 =& {n+r+1 \choose n} \sum_{j=0}^{r-1} (-1)^{r-1-j} {r+1 \choose j} L^n\\
\gamma_1=& {n+r+1 \choose n-1} \sum_{j=0}^{r} {r+2 \choose j}  (-1)^{r-j} (-c_1(\m{S})+(r-j)\m{Q}).L^{n-1} \\
\gamma_2=& {n+r+1 \choose n-2}\sum_{j=0}^{r+1} {r+3 \choose j} (-1)^{r+1-j}  \left(c_1(\m{S})^2-c_2(\m{S}))-(r+1-j)c_1(\m{S}).\m{Q} + \frac{(r+1-j)(r-j)}{2} \m{Q}^2\right).L^{n-2}.
\end{align*}

We have one final expansion to consider:
\begin{align*}
\left\langle\left(c_1(X) +q \right)\cdot ([\omega_k]+1)^{n+r} \cdot (s(H \otimes \pi^*\m{Q})), [\bb{P}(\m{S})] \right\rangle = \sum_{i=0}^n \delta_i k^{n-i}
\end{align*}
The strategy is the same as the above. For the leading order term $\delta_0$, we have
\begin{align*}
\delta_0 =&{n+r \choose n} \sum_{j=0}^r \sum_{i=0}^{n+r-1} (-1)^i  {r \choose j} \left\langle\left( (r+1)H^{i+j+1} + H^{i+j}\right).L^n, [\bb{P}(\m{S})] \right\rangle.
\end{align*}
This contributes only when the power of $H$ is $r-1$. Using this, after some manipulation, we obtain that
\begin{align*}
\delta_0 =&{n+r \choose n} \left( (r+1) \sum_{j=0}^{r-2} (-1)^{r-2-j} {r \choose j} + \sum_{j=0}^{r-1} (-1)^{r-1-j}  {r \choose j} \right)L^n 
\end{align*}

The computations of $\delta_1$ and $\delta_2$ are longer, but similar. In the end, we obtain that
\begin{align*}
\delta_0 =& {n+r \choose n} \left( (r+1) \sum_{j=0}^{r-2} (-1)^{r-2-j} {r \choose j} + \sum_{j=0}^{r-1} (-1)^{r-1-j}  {r \choose j} \right)L^n \\
\delta_1=&{n+r \choose n-1} \sum_{j=0}^{r-1} {r+1 \choose j} (-1)^{r-1-j} \left( c_1(B)+c_1(\cE) \right).L^{n-1}  \\
&+ {n+r \choose n-1} (r+1) \sum_{j=0}^{r-1} {r+1 \choose j}  (-1)^{r-1-j} (-c_1(\m{S})+(r-1-j) .\m{Q}).L^{n-1}   \\
&+ {n+r \choose n-1} \sum_{j=0}^{r} {r+1 \choose j}  (-1)^{r-j} (-c_1(\m{S})+(r-j)H^{r-1}.\m{Q}).L^{n-1}  \\
\delta_2=& {n+r \choose n-2}(r+1) \sum_{j=0}^{r} {r+2 \choose j}(-1)^{r-j} \left(c_1(\m{S})^2-c_2(\m{S})-(r-j) c_1(\m{S}).\m{Q} + \frac{(r-j)(r-j-1)}{2} \m{Q}^2\right).L^{n-2}  \\
&+ {n+r \choose n-2}\sum_{j=0}^{r} {r+2 \choose j}  (-1)^{r-j} \left(-c_1(\m{S})+(r-j) Q \right).(c_1(B)+c_1(\m{E})).L^{n-2} \\
&+ {n+r \choose n-2} \sum_{j=0}^{r+1} {r+2 \choose j} (-1)^{r+1-j} \left(c_1(\m{S})^2-c_2(\m{S})- (r+1-j) c_1(\m{S}).\m{Q} + \frac{(r+1-j)(r-j)}{2} \m{Q}^2\right).L^{n-2}.
\end{align*}

\subsection{A concrete example: rank 2 vector bundles over a surface}
\label{sec:concreteeg}
We now specialise our discussion to the case when the base has dimension $n=2$ and $r=2$, so that $\cE$ is the extension of a rank $2$ slope stable vector bundle $\m{S}$ by a line bundle $\m{Q}$, over a surface. In this case, one verifies that 
\begin{align*}
  \begin{split}\alpha_0 =& 6L^2 \\
\alpha_1 =& -4 c_1(\cE).L \\
\alpha_2 =& c_1(\cE)^2 - c_2(\cE) \\
\beta_0 =& 9L^2\\
\beta_1 =& 3(c_1(B) -2c_1(\cE)).L \\
\beta_2 =& -c_1(B).c_1(\cE) + 2 c_1(\cE)^2 - 3c_2(\cE)
  \end{split}
  \begin{split}\gamma_0 =& 20 L^2 \\
\gamma_1=& -5(3c_1(\m{S})+2c_1(\m{Q})).L\\
\gamma_2 =& 4(c_1(\m{S})^2 - c_2(\m{S}))+3c_1(\m{S}).c_1(\m{Q})+2c_1(\m{Q})^2\\
\delta_0 =& 24L^2\\
\delta_1 =& 8(c_1(B)+c_1(\cE)).L -28c_1(\m{S}).L-16c_1(\m{Q}).L\\
\delta_2 =& 10(c_1(\m{S})^2-c_2(\m{S})) + 6 c_1(\m{S}).c_1(\m{Q}) + 3c_1(\m{Q})^2 \\
&- (3c_1(\m{S}) + 2c_1(\m{Q}))(c_1(B) + c_1(\cE)).
 \end{split}
\end{align*}

Denoting the expansion of the Donaldson-Futaki invariant as
\[
[\omega_k]^{n+r}\mrm{Fut}(\sigma) = a_0 k^{2n}+a_1 k^{2n-1} + a_2 k^{2n-2} + O(k^{2n-3}),
\]
one can from the relations above easily verify that
\begin{align*} \frac{(n+r)!}{L^2} a_0 =& 0
\end{align*}
and that
\begin{align*}
\frac{(n+r)!}{L^2} a_1 =& 72 \left(\mu_L(\m{S}) - \mu_L(\m{E}) \right).
\end{align*}
This gives the relationship with the slope that we already knew from the results of Ross--Thomas (Proposition \ref{prop:RossThomas}).

Finally, we consider the $O(k^{2n-2})$-term. Since we are ultimately interested in the case when $S$ is a destabilising subbundle, i.e.\ $\mu_L(\m{S}) = \mu_L(\m{E})$, we make this assumption in the computation. 
We then obtain the following simplification, which is straightforward to verify given the above formulae and which we expect holds in general.
\begin{lemma}
Suppose $\mu_L(\m{S}) = \mu_L(\m{E})$. Then $\frac{n+r}{n+r+1}  \beta_1 \gamma_1  - \alpha_1 \delta_1=0$.
\end{lemma}

Thus, under this condition, the $O(k^{2n-2})$-term of the Futaki invariant is given by
\begin{align*}
\frac{(n+r)!}{L^2} a_2 =& \frac{n+r}{n+r+1}  (\beta_0 \gamma_2 + \beta_1 \gamma_1 + \beta_2 \gamma_0) - (\alpha_0 \delta_2 + \alpha_1 \delta_1 + \alpha_2 \delta_0).
\end{align*}
Manipulating the formulae above, we then obtain the following.
\begin{lemma}
If $\mu_L(\m{S}) = \mu_L(\m{E})$, the $a_2$-term in the expansion of the Donaldson-Futaki invariant is given by
\begin{align*} 
\frac{(n+r)!}{L^2} a_2 =& \frac{2}{5} \left( -51 c_1(\m{S})^2 + 78 c_2(\m{S}) -3 c_1(\m{S}).c_1(\cE) +41c_1(\cE)^2 -60c_2(\cE)\right) \\
&+12 \left(\mu_{c_1(B)}(\m{S}) - \mu_{c_1(B)}(\cE)\right).
\end{align*}
\end{lemma}

\begin{remark}
The formula for the expansion of the Donaldson-Futaki invariant can be compared to that of Keller and Ross in \cite{KellerRoss14}, who considered the case when the rank of $E$ is $2$ instead of $3$. Their formula is obtained through the Ross-Thomas expansion, but can also be recovered from the localisation approach used here, using the formulae from the previous section with $r=1$.   Keller--Ross  relate the (Chow) stability of $\bb{P}(\cE)$ to the Gieseker stability of $\cE$. Note that in general, adiabatic slope stability cannot be equivalent to Gieseker stability since adiabatic slope stability does not have any terms involving the higher Chern classes of the base $B$, but such terms do appear in the Todd class, and therefore in Gieseker stability. When $B$ is a surface, the $c_2(B)$-term appearing in Gieseker stability does depend on the bundle, and does not play a role.
\end{remark}

\subsubsection{An adiabatically slope stable vector bundle over the blow up of $\bb{P}^2$}
\label{sec:concreteexample}
We will now apply the above to a concrete example. Let $B=\Bl_p \mathbb{P}^2$. The Picard group is generated by the pullback of the $\mathcal{O}(1)$ line bundle on $\mathbb{P}^2$ and the exceptional divisor $D$ of the blowup. Note that the canonical divisor $K_B=-3H+D$, thus $L:=3H-D$ is ample on $B$.

Let $\m{S}$ be the pullback of a slope stable vector bundle on $\mathbb{P}^2$ with $c_1(\m{S})=0$ and let $\m{Q}=H-3D$. Then
\begin{align*}
\mu_L(\m{Q}) =& (3H-D)(H-3D) = 0, \\
\mu_L(\m{S}) =&0.
\end{align*}
By the Bogomolov inequality, $c_2(\m{S}) \geq 0$. In fact, over projective spaces, simplicity is equivalent to slope stability for rank $2$ vector bundles (see \cite[Theorem 2.1.2.10]{okonek2011vector}), and for each $c_2>0$ there exists a simple bundle with $c_1(\m{S})=0$ \cite[Theorem 8]{Schwarzenberger61}. For our purposes, we can take $c_2(\m{S}) = 1$.

Assume that $\cE$ is a simple extension 
\begin{align}
\label{eq:Eexactseq}
0 \to \m{S} \xrightarrow{i} \cE \xrightarrow{q} \m{Q} \to 0
\end{align}
of $\m{S}$ by $\m{Q}$ and let $\m{E}_0 = \m{S} \oplus \m{Q}$ be the trivial extension. Then
\begin{align*}
c_1(\cE) =& c_1(\m{Q}) \\
c_2(\cE) =& c_2(\m{S}).
\end{align*}
Now, this gives that $a_0=a_1=0$. We also have that
\begin{align*}
c_1(\cE)^2 =& (H-3D)^2 =-8 \\
\mu_{c_1(\m{B})}(\cE) =& \frac{1}{3} (3H-D)(H-3D) =0.
\end{align*}
Thus
\begin{align*}
\frac{5!}{2L^2} a_2 =&  18 c_2(\m{S}) +41(-3)\\
=& 18c_2(\m{S}) -328.
\end{align*}
Since $c_2(\m{S}) = 1 < \frac{328}{18}$, $\cE$ is adiabatically slope stable, provided such an extension with $\cE$ being simple exists.

It remains to show that there exists in fact a simple extension of $\m{S}$ by $\m{Q}$.
To do this, we first compute $h^1(\m{S} \otimes \m{Q}^*)$ in order to show that there are non-trivial extensions of $\m{S}$ by $\m{Q}$.
We then show that any non-trivial extension has to be simple. 

We have
\begin{align*}
h^1(\m{S} \otimes \m{Q}^*) =& h^0(\m{S} \otimes \m{Q}^*) + h^2(\m{S} \otimes \m{Q}^*) - \chi(\m{S} \otimes \m{Q}^*) \\
=&  h^0(\m{S} \otimes \m{Q}^*) + h^0(K_B \otimes \m{S}^* \otimes \m{Q})- \chi(\m{S} \otimes \m{Q}^*)
\end{align*}
where in the second line we used Serre duality. Now, $\m{S} \otimes \m{Q}^*$ is a stable bundle with slope zero, and so cannot admit a section (a section would give an inclusion $\mathcal{O} \subset \m{S} \otimes \m{Q}^*$, whose slope would then agree with that of $\m{S} \otimes \m{Q}^*$). As $K_B=-3H+D$ and $\m{S}^* \otimes \m{Q}$ is a stable bundle with vanishing slope, $ K_B \otimes \m{S}^* \otimes \m{Q}$ is a stable bundle with negative slope, and so cannot admit a section either. Thus $h^0(\m{S} \otimes \m{Q}^*)=0=h^0(K_B \otimes \m{S}^* \otimes \m{Q})$ and so
\begin{align*}
h^1(\m{S} \otimes \m{Q}^*) =& - \chi(\m{S} \otimes \m{Q}^*).
\end{align*}

To compute $h^1$, we then use Riemann--Roch. Note that when $B=\Bl_p \mathbb{P}^2$, we have
$$
\int_B c_1(B)^2+c_2(B) = 12.
$$ 
Moreover,
\begin{align*}
c_1(\m{S} \otimes \m{Q}^*) =&c_1(\m{S}) + 2 c_1(\m{Q}^*) = -2(H-3D) \\
c_2(\m{S} \otimes \m{Q}^*) =&c_2(\m{S}) + c_1(\m{S})c_1(\m{Q}^*)+c_1(\m{Q}^*)^2 = c_2(\m{S}) + (H-3D)^2.
\end{align*}
This gives that
\begin{align*}
h^1(\m{S} \otimes \m{Q}^*) &= - \int_B \textnormal{ch}(\m{S} \otimes \m{Q}^*) \textnormal{Td}(B) \\
&
\begin{multlined}
=-\int_B \left(2+c_1(\m{S} \otimes \m{Q}^*) + \frac{c_1(\m{S} \otimes \m{Q}^*)^2 - 2c_2(\m{S} \otimes \m{Q}^*)}{2} \right)\\ \cdot\left( 1+ \frac{c_1(B)}{2} + \frac{c_1(B)^2+c_2(B)}{12} \right)
\end{multlined} \\
&= -(2 - (H-3D)(3H-D)+2(H-3D)^2- c_2(\m{S})-(H-3D)^2) \\
&= 6+c_2(\m{S}) \\
&\geq 6.
\end{align*}
Note that in the above we also showed that $h^2(\m{S} \otimes \m{Q}^*) =h^0(K_B \otimes \m{S}^* \otimes \m{Q})=0$, so that all deformations are unobstructed. Thus there are non-trivial extensions of $\m{S}$ by $\m{Q}$. 

It remains to show that these are simple. More precisely, we claim that all but the trivial extension is simple, i.e.\ we show that $H^0(\End \cE)$ is one dimensional whenever $\cE$ is a non-trivial extension of $\m{S}$ by $\m{Q}$.

We show first that $\cE$ does not split as a sum $\m{S}'\oplus \m{Q}'$ for any other $\m{S}'$ and $\m{Q}'$. Assume by contradiction that it does. Note then that since $\m{S}'$ is both a subsheaf and a quotient of $\cE$, and $\cE$ is semistable, $\m{S}'$ must have the same slope as $\cE$. Similarly, $\m{Q}'$ also has the same slope. Thus $\m{S}'$ and $\m{Q}'$ must be slope semistable, and of the same slope as $\m{S}, \m{Q}$ and $\cE$. 

Next, consider the diagram
\[
\begin{tikzcd}
            &                                       & 0 \arrow[d]                                  &             &   \\
            &                                       & \m{S}' \arrow[d, "i'"] \arrow[rd]                &             &   \\
0 \arrow[r] & \m{S} \arrow[r, "i"] \arrow[r] \arrow[rd] & \cE=\m{S}'\oplus \m{Q}' \arrow[r, "q"] \arrow[d, "q'"] & \m{Q} \arrow[r] & 0 \\
            &                                       & \m{Q}' \arrow[d]                                 &             &   \\
            &                                       & 0                                            &             &  
\end{tikzcd}
\]
The maps $q'\circ i : \m{S} \to \m{Q}'$ and $q\circ i': \m{S}' \to \m{Q}$, obtained by composing with our starting sequence \eqref{eq:Eexactseq}, are morphisms of semistable sheaves. Then by \cite[Corollary to Lemma 1.2.8]{okonek2011vector}, as $\m{S}$ is stable, $q'\circ i$ is either $0$ or injective, and similarly $q\circ i'$ is either $0$ or surjective as $\m{Q}$ is stable. If both $q'\circ i$ is injective and $q\circ i'$ is surjective, then $\rk \m{S} \leq \rk \m{Q}'$ and $\rk \m{S}' \geq \rk \m{Q}$. This gives that
\begin{align*}
\rk \m{E} =& \rk \m{S} + \rk \m{Q} \\
\leq& \rk \m{Q}' + \rk \m{S}' \\
=& \rk \m{E},
\end{align*}
and so we must have equality in both of the inequalities. In particular, in this case both maps are isomorphisms. However, since the sequence \eqref{eq:Eexactseq} does not split by assumption, the morphisms cannot both be isomorphisms. Therefore, at least one of them is the $0$ morphism.

Assume that $q\circ i' = 0$. Then we have two sequences:
\[
\begin{tikzcd}
0 \arrow[r] & \m{S}' \arrow[r, "i'"] \arrow[d, "0"] & \cE \arrow[r, "q"] \arrow[d, "="] & \m{Q} \arrow[r] \arrow[d, "="] & 0 \\
0 \arrow[r] & \m{S} \arrow[r, "i"]                  & \cE \arrow[r, "q"]                & \m{Q} \arrow[r]                & 0
\end{tikzcd}
\]
where the map $\m{S}'\to \m{S}$ is induced by the inclusion $i'(\m{S}) \subseteq \ker(q) = i(\m{S})$, and it is the zero map because $\m{S}$ and $\m{S}'$ are both stable and non isomorphic. Therefore $i'=0$ and analogously, if $q' \circ i =0$, we get that $q'=0$.
Therefore $\cE$ does not split.

Next we show that $h^0(\End \cE) = 1$.
Taking the dual of the sequence \eqref{eq:Eexactseq} and tensoring with $\m{E}^*$, we obtain the sequence
\begin{align*}
0 \to \cE \otimes \m{Q}^* \to \End \cE \to \cE \otimes \m{S}^* \to 0.
\end{align*}
This gives a long exact sequence
\begin{align*}
0 \to H^0(\cE \otimes \m{Q}^*) \to H^0 (\End \cE) \to H^0(\cE \otimes \m{S}^*) \to \ldots.
\end{align*}

First we observe that since the sequence \eqref{eq:Eexactseq} does not split and $\m{Q}$ is stable, $H^0(\cE \otimes \m{Q}^*) = \{0\}$. Indeed, a section $\varphi \in H^0(\cE \otimes \m{Q}^*)$ would induce an endomorphism $q \circ \varphi$ of $\m{Q}$. From the stability of $\m{Q}$, we have that $q \circ \varphi$ is either the identity or the zero map. Since the sequence \eqref{eq:Eexactseq} does not split, it cannot be the identity, thus it is the zero map. If $\varphi\ne0$, then $\varphi$ is injective since $\m{Q}$ is stable, and so we have a short sequence
\[
0 \to \m{Q} \xrightarrow{\varphi} \cE \xrightarrow{q} \m{Q} \to 0.
\]
This implies that $\cE$ splits, against our previous conclusion. So $\varphi=0$.

Next, since $\m{S}$ is stable, $H^0(\m{S} \otimes \m{S}^*) = \langle \Id_{\m{S}}\rangle$. Moreover, since $\m{S}$ is a subbundle of $\cE$, we have an inclusion $H^0(\m{S} \otimes \m{S}^*) \subseteq H^0(\cE \otimes \m{S}^*)$. We claim that this inclusion is an equality. Indeed, for the trivial extension, we have
$$
H^0(\cE_0 \otimes \m{S}^*) = H^0 (\End \m{S}) \oplus H^0(\m{Q} \otimes \m{S}^*).
$$
Since $\m{Q} \otimes \m{S}^*$ is a stable bundle of slope zero and $\m{Q}$ and $\m{S}$ are not isomorphic, it cannot admit a section. So the inclusion $H^0(\End \m{S}) \subseteq H^0(\cE_0 \otimes \m{S}^*)$ is an equality for the trivial extension $\cE_0$. Since $\cE_0 \otimes \m{S}^*$ has at least as many sections as $\cE \otimes \m{S}^*$, it follows that the inclusion is an equality also for $\cE$.

For the above, we see that 
$$
1 \leq h^0 (\End \cE) \leq h^0(\m{E} \otimes \m{Q}^*)+h^0(\cE \otimes \m{S}^*)=1.
$$
Thus $h^0(\End \cE) = 1$ and so $\cE$ is simple.

\begin{remark}
The above strategy for proving that $\cE$ is indecomposable and simple has been implemented in similar contexts by Brambila-Paz (see for example \cite{brambilapaz1990moduli} and \cite[Lemma 2.5]{BrambilaSierra2022moduli}); we gratefully thank L.\ Brambila-Paz for explaining this to us.
\end{remark}

The upshot is that $\cE$ satisfies all the requirements of Theorem \ref{thm:main}. Unfortunately, the base $B=\Bl_p\mathbb{P}^2$ does not. It has automorphisms and only has an extremal metric, not a cscK one. Luckily, this is easily remedied, since the intersection numbers are unchanged upon pulling back to blowups in points.

First we blow up $\Bl_p\bb{P}^2$ to a base $B'$ with a cscK metric. This is achieved by blowing up two further points away from the exceptional divisor in $\Bl_p\bb{P}^2$ whose line between them does not contain $p$. We then pick the polarisation such that the volumes of all the exceptional divisors equal that of $D$ above, which is just the anticanonical polarisation on the blowup $B'$ of $\bb{P}^2$ in three non-collinear points. Note that $B'$ admits a K\"ahler-Einstein metric \cite{tianyau1987}, so a cscK metric in this polarisation. However, $B'$ still does not satisfy our assumptions on the base, as it has automorphisms. We therefore further blow up $B'$ to additionally make the automorphism group trivial, which is a generic condition once one blows up sufficiently many points. Note that it is not true that all blowups of $B'$ admit cscK metrics. However, the Arezzo--Pacard theorem \cite{arezzopacard1, arezzopacard2} guarantees that we can find a collection of points in $B'$ such that the resulting blown up manifold does not have any continuous automorphisms and still has a cscK metric, in a certain polarisation. To apply their results \cite[Theorem 1.3]{arezzopacard2}, two conditions need to be verified. These are that the moment map for the isometry group action evaluated at the points blown up spans the dual of the Lie algebra of the isometry group, and that a certain weighted average, depending on the ratio of the volumes of the exceptional divisors, of these values vanishes. The first condition is generic \cite[Lemma 1.1]{arezzopacard2} once one blows up sufficiently many points, and the latter can be seen to hold by making a suitable choice for the ratio between the volumes of the exceptional divisors.  Replacing the base $\Bl_p \bb{P}^2$ with $B$ where $\tau : B \to \Bl_p \bb{P}^2$ is the blowup in the above collection of points, we obtain the same expansion of the Donaldson-Futaki invariant for the pulled back bundle (up to the change in the volume of the polarisation). Moreover, by \cite{sibley15}, $\tau^*E$ remains slope semistable and has a graded object which is the pullback of the graded object before blowing up. We can then apply our main result to obtain the following, which is a more detailed version of Corollary \ref{cor:eg}.
\begin{corollary}
\label{cor:eg2}
Let $\tau : B \to \Bl_p \bb{P}^2$ be the blowup of $\Bl_p \bb{P}^2$ in a collection of points as above and let $L$ be a polarisation which makes two of the points have the same volume as $D \subset \Bl_p \bb{P}^2$, and the remaining exceptional divisors have sufficiently small volume. Then $\tau^*\cE$ is strictly slope semistable, and $X=\mathbb{P}(\tau^*\cE)$ admits a cscK metric in $c_1(\mathcal{O}(1) \otimes L^k)$ for all sufficiently large $k$.
\end{corollary}

\begin{remark}
One can extend a stable bundle $\m{S}$ by the structure sheaf $\mathcal{O}$ on $\bb{P}^2$ to produce a strictly slope semistable rank $3$ bundle on $\bb{P}^2$. As above, any non-trivial extension $\cE$ will then be a simple bundle. However, in this case, as $c_1(\cE)  =0$,
\begin{align*}
\frac{5!}{2L^2} a_2 =&  18 c_2(\m{S}) \geq 0,
\end{align*}
and thus $\cE$ will be asymptotically slope unstable, no matter which $\m{S}$ we choose.
\end{remark}

\begin{remark}
In \cite{Hattori2022}, Hattori introduced a notion of stability for fibrations called \emph{adiabatic K-stability}. He defines a fibration $(X,H) \to (B,L)$ to be adiabatically K-stable if there exists an $\epsilon_0$ such that the Donaldson-Futaki invariant of all test configurations for $(X, \epsilon H+ L)$ is positive for all $\epsilon \in (0,\epsilon_0)$. He defines uniform adiabatic K-stability by requiring that there exists a $\delta>0$ and an $\epsilon_0 >0$ such that for all $\epsilon \in (0, \epsilon_0)$,
\[
\mrm{DF}_{(X, \epsilon H + L)} (\m{X},\m{H}) \ge \delta \norm{(\m{X}, \m{H})}_m
\]
for all test configurations $(\m{X}, \m{H})$ for $(X, \epsilon H + L)$, where $\delta$ does not depend on $\epsilon$ and $\norm{(\m{X}, \m{H})}_m$ is the minimum norm of the test configuration \cite{dervan2016uniformstability, bhj17}.
In terms of our adiabatic parameter $k = \epsilon^{-1}$, this means that the fibration is uniformly adiabatically K-stable if there exists a $\delta >0$ and a $k_0>0$ such that
\begin{align}
\label{eq:adiabaticKstab}
\mrm{DF}_{(X, H + kL)}(\m{X},\m{H}) \ge \frac{\delta}{k} \norm{(\m{X}, \m{H})}_m
\end{align}
for all $k>k_0$. Note that we divide by $k$ on the right-hand side because the Donaldson--Futaki invariant scales like the reciprocal of the scaling factor upon scaling the underlying polarisation, but the minimum norm remains unchanged.

Let now $\m{E}$ be the vector bundle in Corollary \ref{cor:eg2} and $(\m{X}_{\m{S}}, \m{H}+kL)$ be a test configuration induced by the vector subbundle $\m{S}$ used to construct $\m{E}$. Then the minimum norm admits an expansion \cite[\S 2.5]{dervansektnan19b}:
\begin{equation}\label{eq:minimalnorm}
\norm{(\m{X}_{\m{S}}, \m{H}+kL)}_m = k^n\left( \frac{L^n. \m{H}^r}{r+1} + L^n.\m{H}^r.(\m{H}-H) \right) + \Ok{n-1}.
\end{equation}
This leading order term is the Donaldson-Futaki invariant of the induced test configuration for the generic fibre, which is a non-trivial product test configuration for $\bb{P}^r$. 
From the expression \eqref{eq:DFexpansion_RossThomas}, we see that the first non-vanishing term of the Donaldson-Futaki invariant is the $k^{n-2}$-term. However, from the expression \eqref{eq:minimalnorm} we see that the leading order term of the minimum norm is of order $k^n$ for some non-trivial test configuration. Thus, in this case, the left-hand side in the inequality \eqref{eq:adiabaticKstab} is $O(k^{n-2})$, but the right-hand side is $O(k^{n-1})$. In particular, $\bb{P}(\m{E})$ cannot be uniformly adiabatically stable.

The upshot is that with $\m{E} \to B$ as above, $\bb{P}(\m{E})$ admits a cscK metric, and is therefore an adiabatically K-stable manifold that is not uniformly adiabatically K-stable. The same considerations apply to any strictly semistable vector bundle that is adiabatically slope stable, over a base with discrete automorphism group that admits a cscK metric. Note, on the other hand, that if $\m{E}$ is stable, then the left and right-hand sides of \eqref{eq:adiabaticKstab} are both $O(k^{n-1})$ when $(\m{X}, \m{H})$ arises from a subbundle $\m{S}$, so the construction of Hong \cite{hong98} does not yield this phenomenon. It therefore seems that uniform adiabatic K-stability is more directly related with the optimal symplectic connection equation rather than the cscK equation on the total space.
We thank M.\ Hattori for pointing out this feature to us.
\end{remark}

\bibliography{osc}
\bibliographystyle{amsplain}

\end{document}